\documentclass{amsart}

\usepackage{latexsym}
\usepackage{enumerate}
\usepackage{amssymb}
\usepackage{amsfonts}
\usepackage{amsmath}
\usepackage{mathrsfs}
\usepackage{amsthm}
\usepackage{geometry}
\usepackage[colorlinks,
linkcolor=red,
anchorcolor=blue,
citecolor=green
]{hyperref}
\usepackage{varioref}
\usepackage{hyperref}
\usepackage{cleveref}
\usepackage[all]{xy}
\usepackage{tocvsec2}

\newcommand{\Ric}{\mathrm{Ric}}

\newcommand{\Vol}{\mathrm{Vol}}

\newcommand{\sq}{\backslash}

\newcommand{\Aut}{\mathrm{Aut}}

\newcommand{\sphere}{{\mathbb{S}}}
\newcommand{\proj}{\pi_\mathrm{pr}}
\newcommand{\blowup}{\pi_\mathrm{b}}

\newtheorem{thm}{Theorem}[section]

\newtheorem{fact}{Fact}[section]
\newtheorem{prop}{Proposition}[section]

\newtheorem{lmm}{Lemma}[section]
\newtheorem{question}{Question}[section]

\theoremstyle{remark} 
\newtheorem*{rmk}{Remark}

\title[$4$-Manifold with Prescribed Asymptotic Cone]{A family of $4$-Manifolds with Nonnegative Ricci Curvature and Prescribed Asymptotic Cone}

\author{Shengxuan Zhou}
\address{Beijing International Center for Mathematical Research\\
Peking University\\
Beijing\\ 100871\\ China}
\email{zhoushx19@pku.edu.cn}

\thanks{}
\keywords{}
\date{}
\dedicatory{}

\begin{document}

\begin{abstract}
In this paper, we show that for any finite subgroup $\Gamma 
< O(4)$ acting freely on $\mathbb{S}^3$, there exists a $4$-dimensional complete Riemannian manifold $(M,g)$ with $\Ric_g \geq 0 $, such that the asymptotic cone of $(M,g)$ is $C(\mathbb{S}_\delta^3 /\Gamma )$ for some $\delta = \delta (\Gamma ) >0$. This answers a question of Bru\`e-Pigati-Semola \cite{bps1} about the topological obstructions of $4$-dimensional non-collapsed tangent cones. Combining this result with a recent work of Bru\`e-Pigati-Semola \cite{bps1}, one can classify the $4$-dimensional non-collapsed tangent cone in the topological sense.
\end{abstract}
\maketitle
\tableofcontents

\section{Introduction}
\label{introduction}

Let $(M_i ,g_i ,p_i)$ be a sequence of $n$-dimensional complete Riemannian manifolds with a uniform lower bound on the Ricci curvature $\Ric_{g_i} \geq -(n-1) $. Gromov’s compactness theorem then implies that there exists a subsequence of $(M_i ,g_i ,p_i)$ converging to a metric space $(X_\infty ,d_{\infty} ,p_\infty )$. Without loss of generality, we assume that $(M_i ,g_i ,p_i)$ converge to $(X_\infty ,d_{\infty} ,p_\infty )$.

In this paper, we consider non-collapsed case, where $\Vol (B_1 (p_i)) \geq v >0 $. In this case, the works of Colding \cite[Theorem 0.1]{colding1} and Cheeger-Colding \cite[Theorem 5.9]{chco3} show that the Hausdorff dimension of $X$ is $n$, and the measures of $(M_i,g_i)$ converge to the $n$-dimensional Hausdorff measure $\mathcal{H}^n $ on $(X^n ,d )$. 

In the non-collapsed case, there are many properties derived from the study of tangent cones. For tangent cones isomorphic to the Euclidean space, a Reifenberg type result of Cheeger-Colding \cite[Theorem 5.14]{chco3} shows that the limit space has a manifold structure away from a codimension $2$ set, which is the complement of an open neighborhood of $R_n (X)$, where $R_n (X)$ is the subset of $X$ containing points $x\in X$ satisfying that every tangent cone of $X$ at $x$ is isometric to $\mathbb{R}^n$. More recently, the work of Cheeger-Jiang-Naber \cite[Theorem 1.14]{jcwsjan1} was able to prove that in this case the codimension $2$ Hausdorff measure of the subset we removed can be finite.

A natural problem is to consider those tangent cones that are not isomorphic to a Euclidean space. For $4$-dimensional non-collapsed Ricci limit space with $2$-side Ricci bound and an $L^2$-integral bound of curvature, a standard result due to Anderson \cite{anderson1}, Bando–Kasue–Nakajima \cite{bkn1} and Tian \cite{tg5} shows that $X$ is a topological orbifold with isolated singularities, and the tangent cones at the singular points are $C(\mathbb{S}^3 /\Gamma )$ for some $\Gamma < O(4) $. Note that Cheeger-Naber have shown that the $L^2$-integral bound on curvature is automatic. Recently, Bru\`e-Pigati-Semola \cite{bps1} proved that any tangent cone on a $4$-dimension non-collapsed Ricci limit space is homeomorphic to $C(\mathbb{S}^3 /\Gamma )$ for some $\Gamma < O(4) $ acting freely on $\mathbb{S}^3$.

Another question is which cones are tangent cones of Ricci limit spaces. If $\Gamma$ is a subgroup of $SU(2)$, then Kronheimer \cite{krhm1} proved that there exists a Ricci flat $4$-manifold with Euclidean volume growth whose cone at infinity is isometric to $C(\mathbb{S}^3/\Gamma)$. By considering quotients of Kronheimer's examples, \c Suvaina \cite{sus1} and Wright \cite{wright1} extended this result to the case when $\Gamma$ is a cyclic group of type $\frac{1}{dn^2} (1,dnm-1) $. Note that there are numerous other discrete subgroups of $O(4)$. Moreover, Biquard \cite{biq1, biq2} and Ozuch \cite{ozuch1, ozuch2} obtained some obstructions to the existence of Einstein metrics with positive Einstein constant on topological desingularization of $4$-dimensional orbifolds. Thus, the following question arises from Bru\`e-Pigati-Semola \cite{bps1}.

\begin{question}[{\cite{bps1}, Question 1.15}]
Let $\Gamma <O(4)$ be a discrete group acting freely on $\mathbb{S}^3$. Is there an $RCD(2,3)$ metric over $ \mathbb{S}^3 /\Gamma $ such that $C(\mathbb{S}^3 /\Gamma )$ is a non-collapsed Ricci limit space?
\end{question}

In this paper, we investigate this question. The following theorem is our main theorem.

\begin{thm}\label{thmconstructionO4}
Let $\Gamma $ be a finite subgroup of $O(4)$ acting freely on the unit sphere $\mathbb{S}^3$. Then there exists a complete Riemannian manifold $(M,g)$ such that $\Ric_g \geq 0 $, and $(M,g)$ is asymptotic to the cone $C(\mathbb{S}_\delta^3 /\Gamma )$ for some $\delta = \delta (\Gamma ) >0$, where $\mathbb{S}_\delta^3 $ is the round $3$-sphere with radius $\delta$.
\end{thm}

Combining Theorem \ref{thmconstructionO4} with a recent work of Bru\`e-Pigati-Semola \cite{bps1}, one can classify the $4$-dimensional non-collapsed tangent cones in the topological sense.

Our construction relies on the following observation.

\begin{itemize}
    \item Fundamental groups of spherical $3$-manifolds have fixed point free $U(2)$-representations \cite{thur1}.
    \item There exists a surgery-based method for resolving the quotient singularities \cite{kl1}.
    \item From Perelman \cite{perel1}, Menguy \cite{menguy2}, to Colding-Naber \cite{coldingnaber2}, a large number of methods for constructing manifolds with non-negative Ricci curvature have been developed.
\end{itemize}

In fact, the subgroup in $O(4)$ acting freely on $\mathbb{S}^3$ must be conjugated to a subgroup of $U(2)$ in $O(4)$. Therefore, we can use the resolution of quotient singularities. Our key step is to estimate the lower bound of the Ricci curvature concerning the surgery version of the desingularization. Moreover, the manifold we construct is precisely the minimal resolution of $\mathbb{C}^2 /\Gamma $.

\begin{prop}\label{propconstructionU2}
Let $\Gamma $ be a finite subgroup of $U(2) \subset O(4)$ acting freely on the unit sphere $\mathbb{S}^3$, and $M$ be the minimal resolution of $\mathbb{C}^2 /\Gamma $. Then there exists a complete Riemannian metric $g$ on $M$ such that $\Ric_g \geq 0 $, and the manifold $(M,g)$ is asymptotic to the cone $C(\mathbb{S}_\delta^3 /\Gamma )$ for some $\delta = \delta (\Gamma ) >0$.
\end{prop}

\begin{rmk}
Similar to the examples by Perelman \cite{perel1} and Menguy \cite{menguy2}, for each $\Gamma$, one can see that there exist infinitely many complete Riemannian manifolds with non-negative Ricci curvature that are asymptotic to $C(\mathbb{S}_\delta^3 /\Gamma)$. Moreover, the manifolds we constructed are diffeomorphic to the resolutions of $\mathbb{C}^2 /\Gamma$. See also the remark below the proof of Theorem \ref{thmconstructionO4}.
\end{rmk}

 Now, let's discuss potential connections between our result and the case of Einstein metrics. For any finite subgroup $\Gamma < O(4)$ acting freely on the unit sphere $\mathbb{S}^3$, one can define
 \begin{eqnarray}
     \delta_\Gamma & = & \sup \left\{ \delta > 0 : \textrm{ $ \exists\; (M^4,g)$ with $\Ric_g \geq 0$ that is asymptotic to $C(\mathbb{S}_\delta^3 /\Gamma )$. } \right\}
 \end{eqnarray}
Then $\delta_\Gamma \leq 1 $, and Theorem \ref{thmconstructionO4} is equivalent to the condition $\delta_\Gamma >0$ holds for any finite subgroup $\Gamma < O(4)$ acting freely on the unit sphere $\mathbb{S}^3$. 

Furthermore, a Ricci flat manifold with Euclidean volume growth is always asymptotic to $C(S^3_1/\Gamma)$ for some $\Gamma<O(4)$. Thus, by results from Kronheimer \cite{krhm1}, \c Suvaina \cite{sus1}, and Wright \cite{wright1}, we know that when $\Gamma$ satisfies one of the following conditions, $\delta_\Gamma=1$: 
\begin{enumerate}[\rm (I).]
    \item $\Gamma<SU(2)$,
    \item $\Gamma = \left< \left(\begin{array}{cc}
e^{\frac{2\pi \sqrt{-1}}{dn^2} } & 0 \\
0 & e^{\frac{2(dmn -1)\pi \sqrt{-1}}{dn^2} }
\end{array}\right) \right> < U(2) $ for $d ,m \geq 1 $, $n\geq 2$ and $n,m$ relatively prime. 
\end{enumerate}
To the best of the author's knowledge, examples of Ricci flat manifolds with Euclidean volume growth beyond those mentioned are currently unknown, and their existence remains a long-standing open question in differential geometry.

Recently, Liu \cite{lg1} has proven, under some additional assumptions, that when a manifold $(M,g)$ has non-negative Ricci curvature and one of its asymptotic cones is Ricci flat, then $\Ric_g =0$. This has inspired us to pose the following question, which might be considered as a quantitative version of the question about the existence of Ricci flat $4$-manifolds with Euclidean volume growth.

\begin{question}
Let $\Gamma <O(4)$ be a discrete group acting freely on $\mathbb{S}^3$. 
\begin{itemize}
    \item Can one calculate $\delta_\Gamma $ for a given $\Gamma$? In particular, for which $\Gamma$ does $\delta_\Gamma =1$?
    \item If $\delta_\Gamma =1$, is there a Ricci flat manifold $(M^4,g)$ asymptotic to $C(S^3_1/\Gamma)$?
\end{itemize}
\end{question}

This paper is organized as follows. At first, we recall the resolution of quotient singularities on complex surfaces in Section \ref{sectionblowup}. In Section \ref{sectionpreliminaries}, we collect some formulas related to computing the Ricci curvature. Then we will establish the resolution surgery in the cyclic case and non-cyclic case in Section \ref{sectionsurgeryriccicyclic} and Section \ref{sectionsurgeryriccigeneral}, respectively. Finally, the proofs of Theorem \ref{thmconstructionO4} and Proposition \ref{propconstructionU2} are contained in Section \ref{sectionproofmainthm}. For convenience, we recall the existence of fixed point free $U(2)$-representations of spherical $3$-manifold in Appendix \ref{appspherical}.

\textbf{Acknowledgement.} The author wants to express his deep gratitude to Professor Gang Tian for constant encouragement. The author would also be grateful to Professor Wenshuai Jiang and Professor Aaron Naber for helpful discussions and comments. Additionally, the author thanks Daniele Semola for his interests in this note.

\section{Resolution of quotient singularities on complex surfaces}
\label{sectionblowup}

In this section, we recall the resolution of the quotient singularity $0$ on the orbifold $\mathbb{C}^2 /\Gamma$, where $\Gamma \leq U(2)$ acts free on the unit sphere $\mathbb{S}^3$. See also \cite[Chapter IV, \S 5 - \S 9]{kl1} for more details.

We begin by describing the basic surgery in the progress of resolving the quotient singularities. Consider the natural projection $\proj :\mathbb{C}^2 \to\mathbb{C}P^1$. Then the holomorphic line bundle $\mathcal{O}(-1) \to \mathbb{C}P^1 $ is given by $\cup_{x\in\mathbb{C}P^1} \proj^{-1} (x) \to \mathbb{C}P^1 $. Hence the total space of $\mathcal{O}(-1)$ is
\begin{equation*}
\mathrm{Total} (\mathcal{O}(-1)) = \cup_{x\in\mathbb{C}P^1} \proj^{-1} (x) = \left\{ (z_1 ,z_2 , [w_1 ,w_2]) \in \mathbb{C}^2 \times \mathbb{C}P^1 : (z_1 ,z_2 ) \in [w_1 ,w_2] \right\}.
\end{equation*}
Write $\hat{\mathbb{C}}^2 = \mathrm{Total} (\mathcal{O}(-1)) $. Then $(z_1 ,z_2 , [w_1 ,w_2]) \mapsto (z_1 ,z_2 )$ gives a holomorphic map $\blowup : \hat{\mathbb{C}}^2 \to \mathbb{C}^2 $. Moreover, $\blowup^{-1} (0) $ is a holomorphic submanifold of $\hat{\mathbb{C}}^2 $, and $\blowup $ induces a biholomorphic map $\hat{\mathbb{C}}^2 \sq \blowup^{-1} (0) \to \mathbb{C}^2 \sq \{ 0 \} $. Since $\Gamma\leq U(2)$, one can see that the action of $\Gamma $ on $\mathbb{C}^2$ can be lifted to a unique action of $\Gamma$ on $\hat{\mathbb{C}}^2$ such that the action of $\Gamma $ and the map $\blowup$ commute. Let $\hat{\mathbb{C}}^2_\Gamma = \hat{\mathbb{C}}^2 /\Gamma $. Now we discuss the desingularization in two cases according to whether $\Gamma $ is a cyclic group.

\subsection{The case of cyclic group.}
Since $\Gamma $ is cyclic, one can assume that $\Gamma $ is generated by $\gamma_0 $, and 
$$\gamma_0 (z_1 ,z_2) = \left( e^{\frac{2\pi \sqrt{-1}}{n} } z_1 ,e^{\frac{2p\pi \sqrt{-1}}{n} } z_2 \right) ,\;\; \forall (z_1 ,z_2) \in \mathbb{C}^2 ,$$ 
where $1\leq p\leq n -1 $ is an integer coprime with $n$. One can see that $\Gamma $ acts free on $\mathbb{S}^3$.

Consider the holomorphic map $\tilde{f}: \mathbb{C}^2 \to \mathbb{C}^2 $ given by $(z_1 ,z_2) \mapsto (z_1 ,z_2^p ) $, one can obtain a biholomorphic map $f:\mathbb{C}^2 /\Gamma' \to \mathbb{C}^2 /\Gamma $, where $\Gamma' $ is generated by $\gamma_1' $, and 
$$\gamma_1' (z_1 ,z_2) = \left( e^{\frac{2\pi \sqrt{-1}}{n} } z_1 ,e^{\frac{2\pi \sqrt{-1}}{n} + \frac{2\pi \sqrt{-1}}{p} } z_2 \right) ,\;\; \forall (z_1 ,z_2) \in \mathbb{C}^2 .$$ 
Note that the map $z\mapsto z^p$ shows that $\mathbb{C} $ is biholomorphic to $ \mathbb{C}/\mathbb{Z}_p $.

Then we can also defines the space $\hat{\mathbb{C}}^2_{\Gamma'}$ as above, although the action of $\Gamma'$ on $\mathbb{S}^3$ is not free. By a direct calculation, one can see that $(z_1 ,z_2 ,[w_1 ,w_2]) \in \hat{\mathbb{C}}^2 $ is a fixed-point for some $\gamma' \in \Gamma' $ if and only if $z_2 = 0$. Let $U_i = \{ (z_1 ,z_2 ,[w_1 ,w_2]) \in \hat{\mathbb{C}}^2 : w_i \neq 0 \} $. Clearly, $U_i$ is invariant under the action of $\Gamma'$. Now we discuss the action of $\Gamma'$ on $U_i$.
\begin{itemize}
\item Since $w_1 \neq 0 $ on $U_1$, we have a natural biholomorphic map $ \varphi_1 : \mathbb{C}^2 \to U_1 \subset \hat{\mathbb{C}}^2 $ given by $(z,w) \mapsto (z ,zw ,[1 ,w]) $. Then $\gamma_1' (z,zw) = \left( e^{\frac{2\pi \sqrt{-1}}{n} } z ,e^{\frac{2\pi \sqrt{-1}}{n} + \frac{2\pi \sqrt{-1}}{p} } zw \right) $ on $ \mathbb{C}^2 \sq \{ 0 \} \subset \hat{\mathbb{C}}^2$ implies that $ \varphi^{-1}_1 \gamma_1' \varphi_1 (z, w) = \left( e^{\frac{2\pi \sqrt{-1}}{n} } z ,e^{ \frac{2\pi \sqrt{-1}}{p} } w \right) $. Hence $ U_1 / \Gamma' \cong \mathbb{C} /\mathbb{Z}_n \times \mathbb{C} /\mathbb{Z}_p \cong \mathbb{C}^2 $ as complex varieties.
\item Since $w_2 \neq 0 $ on $U_2$, we have a natural biholomorphic map $ \varphi_2 : \mathbb{C}^2 \to U_2 \subset \hat{\mathbb{C}}^2 $ given by $(z,w) \mapsto (zw ,z  ,[w ,1]) $. Then $\gamma_1' (zw,z) = \left( e^{\frac{2\pi \sqrt{-1}}{n} } zw ,e^{\frac{2\pi \sqrt{-1}}{n} + \frac{2\pi \sqrt{-1}}{p} } z \right) $ on $ \mathbb{C}^2 \sq \{ 0 \} \subset \hat{\mathbb{C}}^2$ implies that $ \varphi^{-1}_2 \gamma_1' \varphi_2 (z, w) = \left( e^{\frac{2\pi \sqrt{-1}}{n} + \frac{2\pi \sqrt{-1}}{p} } z ,e^{ \frac{-2\pi \sqrt{-1}}{p} } w \right) $. By the holomorphic map $(z,w) \mapsto (z^n ,w) $, we see that $ U_2 / \Gamma' \cong \mathbb{C}^2 / \hat{\Gamma} $ as complex varieties, where $\hat{\Gamma}' $ is generated by $ \hat{\gamma}_1 (z, w) = \left( e^{ \frac{2n\pi \sqrt{-1}}{p} } z ,e^{ \frac{-2\pi \sqrt{-1}}{p} } w \right) $.
\end{itemize}
It follows that the space $\hat{\mathbb{C}}^2_{\Gamma'}$ has only one singular point $(0,0,[0,1])$ as complex space, and an open neighborhood of $(0,0,[0,1])$ is biholomorphic to an open subset of $\mathbb{C}^2 /\hat{\Gamma} $. Note that $|\hat{\Gamma}| = p < n = | \Gamma | $, so we can apply this surgery again and again to the new singular points, and this process ends after a finite number of steps.

\subsection{The case of non-cyclic group.}
Now we consider the case where $\Gamma$ is not a cyclic group. Now we consider the space $\hat{\mathbb{C}}^2_{\Gamma}$. Let $H\leq \Gamma $ be an abelian subgroup. Then the elements in $H$ have a common characteristic vector $v_H \in\mathbb{C}^2 $. Since the action of $\Gamma $ on the unit sphere is free, the restriction of the action of $H$ on $\mathbb{C} v_H $ gives an embedding $H\to U(1) $, and hence $H$ is cyclic. Then we can assume that $\Gamma_{cent} = \Gamma \cap Z(U(2)) $ is generated by
$$\gamma_Z (z_1 ,z_2) = \left( e^{\frac{2\pi \sqrt{-1}}{n} } z_1 ,e^{\frac{2\pi \sqrt{-1}}{n} } z_2 \right) ,\;\; \forall (z_1 ,z_2) \in \mathbb{C}^2 ,$$ 
where $Z(U(2)) = \left\{ e^{\theta \sqrt{-1} } I_2 ,\; \; \theta \in\mathbb{R} \right\} $ is the center of $U(2)$. 

Write $\Gamma_S = \left\{  \gamma \in \Gamma : \gamma S=S \right\} $, where $S$ is a subset of $ \hat{\mathbb{C}}^2 $. Then for any $x=(0,0,[w_1 ,w_2])$, we have $\Gamma_x = \Gamma_{\proj^{-1} (x)} $. Choose unit vectors $v_1 \in \proj^{-1} (x) $ and $v_2 \perp \proj^{-1} (x)$. Hence we have
$$ \Gamma_{\proj^{-1} (x)} \leq \left\{  \gamma\in\Gamma :\;\; \gamma (v_1) = \left( e^{\alpha \sqrt{-1}} v_1 , e^{\beta \sqrt{-1}} v_2 \right) ,\; \alpha ,\beta\in\mathbb{R} \right\} .$$
Without loss of generality, we can assume that $v_1 = (1,0)$, $v_2 = (0,1)$, and $\Gamma_{\proj^{-1} (x)} $ is generated by 
$$\gamma_x (z_1 ,z_2) = \left( e^{\frac{2\pi \sqrt{-1}}{nq} } z_1 ,e^{\frac{2k\pi \sqrt{-1}}{nq} } z_2 \right) ,\;\; \forall (z_1 ,z_2) \in \mathbb{C}^2 ,$$
where $n|k-1 $, $q \nmid k-1 $ and $k$ is coprime with $q$. Since $\Gamma_{\proj^{-1} (x)} $ is a cyclic group, one can apply the argument in the cyclic case to show that $x$ is a cyclic type singular point of $\hat{\mathbb{C}}^2_{\Gamma}$. Then this process reduces to the case where $\Gamma $ is a cyclic group.

\begin{rmk}
In Appendix \ref{appspherical}, we will see that the non-cyclic group $\Gamma$ we actually consider always satisfies $\Gamma_{cent} = \Gamma \cap Z(U(2)) \neq 0 $, so there is no $(-1)$-curve on the complex surface we finally get. So the complex surface finally get is the minimal resolution of $\mathbb{C}^2 /\Gamma$. See also \cite[Section III-6]{bhpv1}.
\end{rmk}

\section{Preliminaries}\label{sectionpreliminaries}
For the convenience of calculations, we list here some formulas related to computing the Ricci curvature in this section. These formulas are all well-known.

\settocdepth{part}
\subsection{ Riemannian submersion with totally geodesics fibers}
We first recall the formulas for calculating the Ricci curvature related to Riemannian submersion with totally geodesics fibers. Let $(M,g_M)$, $(B,g_B)$ be Riemannian manifolds, and $\pi : M\to B$ be a Riemannian submersion. Assume that for any $b\in B$, the fiber $F_b = \pi^{-1} (b) $ is totally geodesic in $M$.

Let $U,V\in\mathcal{V} =TF\subset TM $ be vertical vector fields on $M$, and $X,Y\in\mathcal{H} =\mathcal{V}^\perp \subset TM $ be horizontal vector fields on $M$. Then one can define a tensor on $M$ by
$$ A( X + U, Y+ V) =  (\nabla_X Y)^\mathcal{V} + (\nabla_X V)^\mathcal{H} ,$$
where $ E^\mathcal{H} $ and $ E^\mathcal{V} $ be the components of $E\in TM$ corresponding to the orthonormal decomposition $TM=\mathcal{H}\oplus \mathcal{V} $. Let $g^\mathcal{H}_M$ and $g^\mathcal{V}_M$ denote the restriction of $g_M$ on $ \mathcal{H}$ and $ \mathcal{V} $, respectively. For any $t>0$, we consider the canonical variation, $g^t_M = g^\mathcal{H}_M + t g^\mathcal{V}_M$, of $g_M$. It is easy to see that the fibers $F_b$ are also totally geodesic in $(M,g_M^t)$. Then we have the following property.

\begin{prop}[{\cite[Proposition 9.70]{be1}}]\label{proptotallygeodesiccalculation}
Let $(M,g_M)$, $(B,g_B)$ be Riemannian manifolds, and $\pi : M\to B$ be a Riemannian submersion. Assume that for any $b\in B$, the fiber $F_b = \pi^{-1} (b) $ is totally geodesic in $M$. Then
\begin{eqnarray}
\Ric_{g^t_M} (U,V) & = & \Ric_{g^\mathcal{V}_M} (U,V) + t^2 g_M (AU,AV) ,\\
\Ric_{g^t_M} (U,X) & = & tg_M ( {\rm div}_{g_B} A(X) ,U ) ,\\
\Ric_{g^t_M} (X,Y) & = & \Ric_{g_B} (X,Y) - 2 g_M (A(X,\cdot ) ,A(Y,\cdot )) ,
\end{eqnarray}
where we regard $\Ric_{g_B}$ and $\Ric_{g^\mathcal{V}_M}$ as tensors on $M$, and $A$ is the tensor corresponding to $\pi : (M,g_M) \to (B,g_B) $ as defined above.
\end{prop}

\begin{rmk}
Let $\{ X_i \} $ be an orthonormal basis of $\mathcal{H}_x \subset T_x M$. Then we have
\begin{eqnarray}
g_M (AU,AV) & = & \sum_i g_M (A(X_i ,U), A(X_i ,V) ) ,\\
g_M (A(X,\cdot ) ,A(Y,\cdot )) & = & \sum_i g_M (A(X ,X_i), A(Y ,X_i) ),\\
 {\rm div}_{g_B} A(X) & = & \sum_{i} \nabla_{X_i} A(X_i ,X) .
\end{eqnarray}
\end{rmk}

\subsection{Curvature equations for distance functions}
Now we recall the formulas for calculating the curvature equations related to distance functions. Let $(M,g_M)$ be a Riemannian manifold, $U$ be an open subset of $M$ and $r : U\to \mathbb{R}$ be a smooth function such that $g_M (\nabla r,\nabla r)=1$. In this case, we say that $r$ is a distance function on $U$. Let $X,Y,Z,W \in TM $ be vector fields on $M$, and $S(X) = \nabla_X \nabla r $. Then ${\rm Hess} r (X,Y) = g_M (S(X) ,Y) $. Let ${\rm Hess}^2 r $ denotes the tensor ${\rm Hess}^2 r (X,Y) = g_M (S^2(X) ,Y) = g_M (S(X),S(Y)) $, and let $H=r^{-1} (t) $ denotes the level set for some $t\in r(U)$. Then the second fundamental form ${\rm II}_H = {\rm Hess} r |_{TH \times TH} $. Hence we have the following equations.
\begin{prop}[{\cite[Theorem 3.2.2-3.2.5]{pp1}}]\label{propcurvaturedistancefunction}
Let $(M,g_M)$ be a Riemannian manifold, $U$ be an open subset of $M$, $r : U\to \mathbb{R}$ be a smooth function such that $g_M (\nabla r,\nabla r)=1$, and $H=r^{-1} (t) $ for some $t\in r(U)$. Let $g_H$ denotes the restriction of $g$ on $H$. Then
\begin{eqnarray}
\mathrm{R}_{g_M} (X,Y,Z,W) & = & \mathrm{R}_{g_H} (X,Y,Z,W) - {\rm II}_H (X,W) {\rm II}_H (Y,Z) + {\rm II}_H (X,Z) {\rm II}_H (Y,W) ,\\
\mathrm{R}_{g_M} ( X,Y,Z,\frac{\partial}{\partial r} ) & = & - ( \nabla_X {\rm II}_H ) (Y,Z) + ( \nabla_Y {\rm II}_H ) (X,Z) ,\\
\mathrm{R}_{g_M} ( \frac{\partial}{\partial r},X,Y,\frac{\partial}{\partial r} ) & = & {\rm Hess}^2 r (X,Y) - ( L_{\frac{\partial}{\partial r}} {\rm Hess} r ) (X,Y) ,
\end{eqnarray}
where $X,Y,Z,W \in TH \subset TM $ be vector fields on $H$.
\end{prop}

\subsection{Double warped products}

Let $\varphi ,\phi $ be smooth nonnegative functions on $[0,\infty )$ such that
\begin{itemize}
\item $\varphi ,\phi $ are positive on $(0,\infty )$,
\item $\phi (0) >0 $, $ \phi^{\mathrm{(odd)}} (0) =0 $,
\item $ \varphi (0) =0 $, $ \varphi' (0)= 1 $, $ \varphi^{\mathrm{(even)}} (0) =0 $.
\end{itemize}
Then we can define a Riemannian metric on $ \mathbb{R}^{m+1} \times \sphere^n $ by
$$ g_{\varphi ,\phi } (r) = dr^2 + \varphi (r)^2 g_{\sphere^m} +  \phi (r)^2 g_{\sphere^n} .$$
See also \cite[Proposition 1.4.7]{pp1} for more details. Let $X_0 =\frac{\partial }{\partial r} $, $X_1 \in T_{\sphere^m} $, and $X_2 \in T_{\sphere^n} $ be unit vectors. Then the Ricci curvature of $g_{\varphi ,\phi }$ can be expressed as following. 

\begin{prop}\label{propwarpedproductriccicurvature}
Let $\varphi ,\phi $ and $ g_{\varphi , \phi }$ be as above. Then the Ricci curvature tensor of $ g_{\varphi , \phi }$ can be determined by
\begin{eqnarray}
\Ric_{g_{\varphi , \phi } } \left(X_0 \right) & = & - \left( m \frac{\varphi''}{\varphi} + n \frac{\phi''}{\phi} \right) X_0 ,\\
\Ric_{g_{\varphi , \phi } } \left(X_1 \right) & = & \left[-  \frac{\varphi''}{\varphi} + (m-1)\frac{1-(\varphi')^2 }{\varphi^2} - n \frac{\varphi' \phi'}{\varphi \phi } \right] X_1 ,\\
\Ric_{g_{\varphi , \phi } } \left(X_2 \right) & = & \left[-  \frac{\phi''}{\phi} +(n-1) \frac{1-(\phi')^2 }{\phi^2} - m \frac{\varphi' \phi'}{\varphi \phi } \right] X_2 .
\end{eqnarray}
\end{prop}

\resettocdepth
\section{Resolution surgery with \texorpdfstring{$\Ric \geq 0$}{Lg} : the cyclic case}
\label{sectionsurgeryriccicyclic}

In this section, we will give a resolution surgery with $\Ric \geq 0$ in the cyclic case. Now we state this surgery as following.

\begin{prop}\label{propinductioncyclic}
Let $\Gamma $ be a non-trivial finite cyclic subgroup of $U(2) \subset O(4)$ acts free on the unit sphere $\mathbb{S}^3$. Then for any $\epsilon >0$, we can find a finite cyclic subgroup $\hat{\Gamma}$ of $U(2) $ acts free on the unit sphere $\mathbb{S}^3$ and a pointed metric space $(\mathcal{X},d,x)$ satisfying the following properties:
\begin{itemize}
\item $|\hat{\Gamma}| < |\Gamma |$,
\item The metric subspace $(\mathcal{X}\sq \{x\} ,d)$ is isometric to a Riemannian manifold $(M,g)$ with nonnegative Ricci curvature,
\item The asymptotic cone of $(\mathcal{X},d)$ is isometric to $C(\mathbb{S}^3_{\delta } /\Gamma )$ for some $\delta =\delta (\Gamma ,\epsilon ) >0$,
\item The metric ball $(B_1 (x) ,d)$ is isometric to $B_1 (o) \subset ( C(\mathbb{S}^3_{\epsilon } /\hat{\Gamma} ) ,o)$.
\end{itemize}
\end{prop}

Topologically, the space in this proposition in just the space given in Section \ref{sectionblowup}, and all we have to do is to construct metrics with $\Ric \geq 0 $ on this space.

\subsection{Warped Berger spheres}
Let us begin with construct a Riemannian metric on the complex orbifold $\hat{\mathbb{C}}^2_{\Gamma'}$ in Section \ref{sectionblowup} such that $\hat{\mathbb{C}}^2_{\Gamma'}$ becomes a smooth Riemannian manifold outside the topological singular point $(0,0,[0,1])$. Recall that $\Gamma' < U(2) $ is generated by 
$$\gamma_1' = \left(\begin{array}{cc}
e^{\frac{2\pi \sqrt{-1}}{n} } & 0 \\
0 & e^{\frac{2 \pi \sqrt{-1}}{n} + \frac{2 \pi \sqrt{-1}}{p} }
\end{array}\right) , $$ 
and $p$ is coprime to $n$. Hence
$$\Gamma' = \left<\left(\begin{array}{cc}
e^{\frac{2\pi \sqrt{-1}}{n} } & 0 \\
0 & e^{\frac{2 \pi \sqrt{-1}}{n} }
\end{array}\right) \right> \oplus \left<\left(\begin{array}{cc}
1 & 0 \\
0 & e^{\frac{2 \pi \sqrt{-1}}{p} }
\end{array}\right) \right> = \Gamma_{n,1,1} \oplus \Gamma_{p,0,1} , $$ 
where $\Gamma_{n,k,l}$ is the cyclic group generated by $\left(\begin{array}{cc}
e^{\frac{2k\pi \sqrt{-1}}{n} } & 0 \\
0 & e^{\frac{2 l\pi \sqrt{-1}}{n} }
\end{array}\right) $.

Note that $\hat{\mathbb{C}}^2_{\Gamma_{n,1,1}} $ is biholomorphic to $\mathrm{Total} (\mathcal{O}(-n))$, the total space of the holomorphic line bundle on $\mathbb{C}P^1 \cong \mathbb{S}^2 $ with Euler number $-n$. Here a holomorphic line bundle is just a plane bundle with holomorphic structure. A classical way to construct smooth Riemannian metrics on $\mathrm{Total} (\mathcal{O}(-n))$ is use the Berger sphere \cite[Exercise 1.6.23]{pp1}. This is very useful in the construction of several interesting examples of non-negative Ricci curvature \cite{coldingnaber2, hnw1, menguy2, perel1}.

But we are considering Riemannian metrics on $\mathrm{Total} (\mathcal{O}(-n)) /\Gamma_{p,0,1} $. By abuse of notation, we use the same notation $\mathcal{O}(-n)$ for the holomorphic line bundle and its total space. If we directly use the Riemannian metrics on $\mathcal{O}(-n) $ constructed by the Berger sphere, then the complex line $\proj^{-1} ([1,0]) =\{ (z,0,[1,0]) : z\in\mathbb{C} \} \subset \hat{\mathbb{C}} $ gives a metric singular set on $( (\mathcal{O}(-n)) /\Gamma_{p,0,1} ) \sq \{(0,0,[0,1]) \}$. A natural idea is to modify the metric on Berger sphere. Now we will describe how to construct a metric on $ \mathcal{O}(-n) /\Gamma_{p,0,1} $ by warping the metric on Berger sphere.

Now we consider the coordinate:
\begin{equation*}
F:\left\{
\begin{aligned}
\relax \left( 0,\frac{\pi}{2} \right) \times ( 0,2\pi )\times (0,2\pi ) \;\;\;\;& \longrightarrow \;\;\;\;\;\;\;\;\;\;\;\;\;\;\;\;\;\;\;\;\;\;\;\; \mathbb{S}^3 ,\\
(\xi ,\alpha ,\beta)  \;\;\;\;\;\;\;\;\;\;\;\;\;\;\;\; &\longmapsto \;\;\;\; \left( \sin\xi e^{\alpha \sqrt{-1}} , \cos \xi e^{(\alpha +\beta )\sqrt{-1}} \right) , \\
\end{aligned}
\right.  
\end{equation*}
and the Hopf fibration $\pi_{Hopf} (z_1 ,z_2) = \left( z_1 \bar{z}_2 , \frac{1}{2} (|z_1|^2 -|z_2|^2 ) \right) $. It follows that 
$$\pi_{Hopf} \circ F (\xi ,\alpha ,\beta) = \left( \frac{\sin (2\xi) e^{-\beta\sqrt{-1}}}{2} , \frac{\cos (2\xi) }{2} \right) .$$ 
Clearly, $\pi_{Hopf}$ is a Riemannian submersion. Moreover, the basic horizontal lift of $D\pi_{Hopf} ( \frac{\partial}{\partial \xi} ) $ and $D\pi_{Hopf} (\frac{\partial}{\partial \beta}) $ is equal to $\frac{\partial}{\partial \xi}$ and $\frac{\partial}{\partial \beta} -\cos^2 \xi \frac{\partial}{\partial \alpha} $, respectively. Hence the horizontal distribution $\mathcal{H}_{Hopf} =\left< \frac{\partial}{\partial \xi} , \frac{\partial}{\partial \beta} -\cos^2 \xi \frac{\partial}{\partial \alpha} \right> $, and for any smooth Riemannian metric $h$ on $\mathbb{S}^2$, the pullback $\pi_{Hopf}^* h$ gives a smooth metric on the vector bundle given by horizontal distribution. 

Then we can construct a smooth metric as following.

\begin{fact}\label{factwarpedbergersphere}
Let $t>0$ be a constant, $f$ be a smooth non-negative function on $\left[ 0,\frac{\pi}{2} \right]$ and $\rho, \phi$ be non-negative smooth functions on $\left[ 0,\infty \right)$. Assume that:
\begin{itemize}
\item $f>0$ on $\left( 0,\frac{\pi}{2} \right)$, $f'(0)=1$, $f'\left( \frac{\pi}{2} \right) = -p$, $f(0)=f^{(even)} (0)=f\left( \frac{\pi}{2} \right)=f^{(even)} \left( \frac{\pi}{2} \right) =0 $,
\item $\rho>0$ on $\left( 0,\infty \right)$, $\rho'(0)=n$, $\rho (0)=\rho^{(even)} (0) = 0 $,
\item $\phi>0$ on $\left[  0,\infty \right)$, $ \rho^{(odd)} (0) = 0 $.
\end{itemize}
Then we have the following metrics.
\begin{itemize}
    \item The metric on $ \left( 0,\frac{\pi}{2} \right) \times ( 0,2\pi )\times (0,2\pi )  $,
$$g_{ f,t} = t^2 d\alpha^2 + \pi^*_{Hopf} h_f ,$$
gives a smooth Riemannian metric on $\mathbb{S}^3 /\Gamma_{n,1,p} $, where $h_f = d\xi^2 + f(\xi)^2 d\beta^2 $ is a metric on $\mathbb{S}^2$ corresponding to the coordinate $\pi_{Hopf} \circ F$.
    \item The metric on $(0,\infty ) \times \left( 0,\frac{\pi}{2} \right) \times ( 0,2\pi )\times (0,2\pi )  $,
$$g_{\rho ,\phi,f} = dr^2 + \rho(r)^2 d\alpha^2 +\phi(r)^2 \pi^*_{Hopf} h_f ,$$
gives a smooth Riemannian metric on $(\mathcal{O}(-n) /\Gamma_{p,0,1} ) \sq \{(0,0,[0,1]) \}$, where $h_f = d\xi^2 + f(\xi)^2 d\beta^2 $ is a metric on $\mathbb{S}^2$ corresponding to the coordinate $\pi_{Hopf} \circ F$.
\end{itemize}

\end{fact}

Now we calculate the Ricci curvature of $g_{ f,t}$ and $g_{\rho ,\phi,f} $. Since we will calculate $\Ric_{g_{\rho ,\phi,f} }$ using the formulas in Proposition \ref{propcurvaturedistancefunction}, which includes $\Ric_{g_{f, \frac{\rho^2}{\phi^2} } }$, we will begin by calculating $g_{ f,t}$. 

Let $U=\frac{\partial}{\partial \alpha}$, $X= \frac{\partial}{\partial \xi} $, and $Y= \frac{1}{f} \left( \frac{\partial}{\partial \beta} - \cos^2 \xi \frac{\partial}{\partial \alpha} \right) $. Then $\{ U,X,Y \}$ is an orthonormal basis on $ \left( \left( 0,\frac{\pi}{2} \right) \times ( 0,2\pi )\times (0,2\pi ) ,g_{f,1} \right) $. By direct calculation, we obtain:
\begin{eqnarray}
 \nabla^{g_{f,1}}_U U=  \nabla^{g_{f,1}}_X X =0, & & \nabla^{g_{f,1}}_U X=  \nabla^{g_{f,1}}_X U =-\frac{\sin (2\xi )}{2f(\xi)} Y, \nonumber \\
  \nabla^{g_{f,1}}_U Y=  \nabla^{g_{f,1}}_Y U =\frac{\sin (2\xi )}{2f(\xi)} X, & & \nabla^{g_{f,1}}_X Y= - \nabla^{g_{f,1}}_Y X =\frac{\sin (2\xi )}{2f(\xi)} U - \frac{f'(\xi)}{2f (\xi)} Y, \nonumber \\
  \nabla^{g_{f,1}}_Y Y = - \frac{f'(\xi)}{2f (\xi)} X & & A(X,Y) = -A(Y,X) = \frac{\sin (2\xi )}{2f(\xi)} U ,\\
  A(X,U) = -\frac{\sin (2\xi )}{2f(\xi)} Y , & & A(Y,U) = \frac{\sin (2\xi )}{2f(\xi)} X ,\nonumber \\
  {\rm div}_{h_f} A(X ) = 0 , & & {\rm div}_{h_f} A(Y ) = \frac{\partial}{\partial \xi} \left( \frac{\sin (2\xi )}{2f(\xi)} \right) U .\nonumber
\end{eqnarray}

By Proposition \ref{proptotallygeodesiccalculation}, we have
\begin{eqnarray}
 \Ric_{g_{f,t}} (U,U) & = & t^2 g_{f,1} (AU,AU) = 2t^2 \left( \frac{\sin (2\xi )}{2f(\xi)} \right)^2 ,\\
 \Ric_{g_{f,t}} (U,X) & = & t g_{f,1} ( {\rm div}_{h_f} A(X ) , U ) =0,\\
 \Ric_{g_{f,t}} (U,Y) & = & t g_{f,1} ( {\rm div}_{h_f} A(Y ) , U ) =t \frac{\partial}{\partial \xi} \left( \frac{\sin (2\xi )}{2f(\xi)} \right) ,\\
 \Ric_{g_{f,t}} (X,X) & = & \Ric_{h_f} (D\pi_{Hopf} (X) , D\pi_{Hopf} (X) ) -2t g_{f,1} ( A_X ,A_X ) \\
 & = &  -\frac{f''(\xi)}{f(\xi)} - 2t \left( \frac{\sin (2\xi )}{2f(\xi)} \right)^2 ,\nonumber\\
 \Ric_{g_{f,t}} (X,Y) & = & \Ric_{h_f} (D\pi_{Hopf} (X) , D\pi_{Hopf} (Y) ) -2t g_{f,1} ( A_X ,A_Y ) = 0 ,\\
 \Ric_{g_{f,t}} (Y,Y) & = & \Ric_{h_f} (D\pi_{Hopf} (Y) , D\pi_{Hopf} (Y) ) -2t g_{f,1} ( A_Y ,A_Y ) \\
 & = & -\frac{f''(\xi)}{f(\xi)} - 2t \left( \frac{\sin (2\xi )}{2f(\xi)} \right)^2 .\nonumber
\end{eqnarray}

Then we calculate the Ricci curvature of $g_{\rho ,\phi,f} $. By definition, we have
\begin{eqnarray}
 {\rm Hess}r & = & \frac{1}{2} L_{\frac{\partial}{\partial r}} g = \rho \rho' d\alpha^2 + \phi \phi' \pi^*_{Hopf} h_f , \nonumber \\
 {\rm Hess}^2 r & = & (\rho' )^2 d\alpha^2 + (\rho' )^2 \pi^*_{Hopf} h_f ,\\
 L_{\frac{\partial}{\partial r}} {\rm Hess}r & = & \left( \rho \rho'' + (\rho' )^2 \right) d\alpha^2 + \left( \phi \phi'' + (\phi' )^2 \right) \pi^*_{Hopf} h_f .\nonumber
\end{eqnarray}
Hence Proposition \ref{propcurvaturedistancefunction} implies that
\begin{eqnarray}
 \Ric_{g_{\rho ,\phi,f} } (\frac{\partial}{\partial r} ,\frac{\partial}{\partial r}) & = & -\frac{\rho''}{\rho} - 2\frac{\phi''}{\phi} ,\\
 \Ric_{g_{\rho ,\phi,f} } (\frac{\partial}{\partial r} ,U) & = & \Ric_{g_{\rho ,\phi,f} } (\frac{\partial}{\partial r} ,X) = \Ric_{g_{\rho ,\phi,f} } (\frac{\partial}{\partial r} ,Y) =0,\\
 \Ric_{g_{\rho ,\phi,f} } (U ,U) & = & \Ric_{g_{f, \frac{\rho^2}{\phi^2} } } (U ,U) - 2\frac{\rho \rho' \phi' }{\phi} - \rho \rho'' = 2\frac{\rho^4}{\phi^4} \left( \frac{\sin (2\xi )}{2f(\xi)} \right)^2 - 2\frac{\rho \rho' \phi' }{\phi} - \rho \rho'' , \\
 \Ric_{g_{\rho ,\phi,f} } (U ,X) & = & \Ric_{g_{f, \frac{\rho^2}{\phi^2} } } (U ,X) =0 ,\\
 \Ric_{g_{\rho ,\phi,f} } (U ,Y) & = & \Ric_{g_{f, \frac{\rho^2}{\phi^2} } } (U ,Y) = \frac{\rho^2}{\phi^2} \frac{\partial}{\partial \xi} \left( \frac{\sin (2\xi )}{2f(\xi)} \right) ,\\
 \Ric_{g_{\rho ,\phi,f} } (X ,X) & = & \Ric_{g_{f, \frac{\rho^2}{\phi^2} } } (X ,X) - \frac{\rho'\phi\phi'}{\rho} - \phi\phi'' -(\phi')^2 \\
 & = & -\frac{f''(\xi)}{f(\xi)} - 2\frac{\rho^2}{\phi^2} \left( \frac{\sin (2\xi )}{2f(\xi)} \right)^2 - \frac{\rho'\phi\phi'}{\rho} - \phi\phi'' -(\phi')^2  ,\nonumber\\
 \Ric_{g_{\rho ,\phi,f} } (X ,Y) & = & \Ric_{g_{f, \frac{\rho^2}{\phi^2} } } (X ,Y) =0 , \\
 \Ric_{g_{\rho ,\phi,f} } (Y ,Y) & = & \Ric_{g_{f, \frac{\rho^2}{\phi^2} } } (Y ,Y) - \frac{\rho'\phi\phi'}{\rho} - \phi\phi'' -(\phi')^2 \\
 & = & -\frac{f''(\xi)}{f(\xi)} - 2\frac{\rho^2}{\phi^2} \left( \frac{\sin (2\xi )}{2f(\xi)} \right)^2 - \frac{\rho'\phi\phi'}{\rho} - \phi\phi'' -(\phi')^2  .\nonumber
\end{eqnarray}

\begin{rmk}
    By letting $f(\xi) = \frac{\sin (2\xi )}{2} $, we can obtain the formulas for the Ricci curvature of the standard Berger spheres. We will use these formulas in Section \ref{sectionsurgeryriccigeneral}.
\end{rmk}

\subsection{Adding edge singularities}
Now we need to find suitable functions $\rho,\phi$ and $f$. In this step, we will construct metrics look like the Riemannian metric with edge singularity along the exceptional divisor $\blowup^{-1} (0) $, so that we can modify the metric near the topological singular point obtained in this step.

At first, we need the following lemma.

\begin{lmm}\label{lemmakappariccipositive}
For any given $\tau \in (0,\frac{1}{10}) $, there exists a constant $\xi_0\in (0,\frac{\tau}{10})$ satisfying the following property. For any $\kappa \geq 2$, we can find a smooth non-negative function $f_\kappa$ on $[0,\frac{\pi}{2}]$ and constants $t_\kappa ,\epsilon_\kappa \in (0,\frac{1}{10})$ such that
\begin{enumerate}[\rm (i).]
    \item $f_\kappa(\xi) = \frac{\sin (\kappa \xi )}{\kappa} $ on $[0,\frac{\xi_0}{\kappa }]$,
    \item $f_\kappa>0 >f_\kappa'' $ on $(0,\frac{\pi}{2})$, and $f_\kappa =c\sin (2\xi ) $ on $[\tau , \frac{\pi}{2} -\tau ]$ for some constant $c>0$,
    \item $f_\kappa'(\frac{\pi}{2}) = -p $, $f_\kappa(\frac{\pi}{2}) = f_\kappa^{ (even)}(\frac{\pi}{2}) = 0 $,
    \item For any $t\in (0,t_\kappa )$, $\Ric_{g_{{f_\kappa},t}} \geq \epsilon_\kappa \left( t^2 d\alpha^2 + \pi_{Hopf}^* h_{f_\kappa} \right) $, where $h_{f_\kappa} = d\xi^2 + f_\kappa (\xi)^2 d\beta^2 $ is a metric on $\mathbb{S}^2$ corresponding to the coordinate $\pi_{Hopf} \circ F$.
\end{enumerate}
\end{lmm}

\begin{proof}
Let 
\begin{eqnarray}
Q_{f,t} = \Ric_{g_{f,t}} - t^2 \left( \frac{\sin (2\xi )}{2f(\xi)} \right)^2 d\alpha^2 + \frac{1}{4} \frac{f''(\xi)}{f(\xi)} \pi^*_{Hopf} h_f .
\end{eqnarray}
Then for any $s_1 ,s_2 ,s_3 \in \mathbb{R}$,
\begin{eqnarray}
 & &  Q_{f,t} (s_1 U +s_2 X+s_3 Y ,s_1 U +s_2 X+s_3 Y ) \nonumber\\
& = & Q_{f,t} (s_1 U +s_2 X ,s_1 U +s_2 X ) + s^2_3 Q_{f,t} (Y, Y )  \\
& = & s_1^2 \left( \frac{\sin (2\xi )}{2f(\xi)} \right)^2 + 2s_1 s_2 \frac{\partial}{\partial \xi} \left( \frac{\sin (2\xi )}{2f(\xi)} \right) + (s_2^2 +s_3^2 ) \left( -\frac{3f''(\xi)}{4f(\xi)} - 2t \left( \frac{\sin (2\xi )}{2f(\xi)} \right)^2 \right) .\nonumber
\end{eqnarray}
Hence $Q_{f,t}\geq 0$ if and only if
\begin{eqnarray} \label{inequalityQft1}
\left(\frac{\partial}{\partial \xi} \left( \frac{\sin (2\xi )}{2f(\xi)} \right)\right)^2 \leq \left( \frac{\sin (2\xi )}{2f(\xi)} \right)^2 \left( -\frac{3f''(\xi)}{4f(\xi)} - 2t \left( \frac{\sin (2\xi )}{2f(\xi)} \right)^2 \right) .
\end{eqnarray}
Now we assume that $ \inf_{\xi\in [0,\frac{\pi}{2}]} \frac{\sin (2\xi )}{2f(\xi)} >0$. Then (\ref{inequalityQft1}) is equivalent to
\begin{eqnarray} \label{inequalityQft2}
\left( 2\cot (2\xi ) -  \frac{f'(\xi)}{f(\xi)} \right)^2 + \frac{3f''(\xi)}{4f(\xi)} \leq - 2t \left( \frac{\sin (2\xi )}{2f(\xi)} \right)^2 .
\end{eqnarray}

Let $\xi_0 \in (0,\frac{\tau}{10}) $ be a positive constant to be specified later, and let $\eta\in C^\infty (\mathbb{R})$ be a cut-off function such that $0\leq \eta\leq 1$, $\eta =1$ on $[0,\frac{\tau}{3}]$, and $\eta=0$ on $[\frac{2\tau}{3} ,\frac{\pi}{2}]$. For each $\kappa \geq 2$, set
\begin{equation}
\hat{f}_{\xi_0 ,\kappa} (\xi) =\left\{
\begin{aligned}
 \frac{\sin (\kappa \xi)}{\kappa}\;\;\;\;\;\;\;\;\;\;\;\;\;\;\;\;\;\;\;\;\;\;\;\;\;\;\;\;\;\;\;\;\;\;\;\;\; &, \textrm{ on $\left[ 0,\frac{2\xi_0}{\kappa } \right] $,}\\
\frac{ \sin (2 \xi_0 ) \sin (2 \xi) }{\kappa \sin (\frac{4\xi_0 }{ \kappa} ) } \left( \frac{2\xi_0}{\kappa} \right)^\alpha \left( \xi^{-\alpha} \eta (\xi) + \left( \frac{\tau}{2} \right)^{-\alpha} (1-\eta (\xi)) \right) &,  \textrm{ on $\left[ \frac{2\xi_0}{\kappa } , \frac{\pi}{2} \right] $,}\\
\end{aligned}
\right.  
\end{equation}
where $\alpha =  \frac{4\xi_0}{\kappa} \cot (\frac{2\xi_0}{\kappa} ) - 4\xi_0 \cot (2\xi_0) $. Clearly, we have $\hat{f}_{\xi_0 ,\kappa} \in C^{1,1} ([0,\frac{\pi}{2}]) $. Note that by the asymptotic expansion of $\cot (\theta )$, one can find a constant $\delta \in (0, \frac{\tau}{10} ) $ such that if $\theta \in (0,\delta ) $, then $\left| \cot (\theta ) - \frac{1}{\theta } \right| \leq \frac{\theta}{2} $. Hence for any $\xi_0 \in (0,\delta ) $, we have $|\alpha| \leq 8\xi_0 $. Assume that $\xi_0 \in (0,\frac{\delta}{10} ) $. We now estimate the left side of inequality (\ref{inequalityQft2}) for the function $\hat{f}_{\xi_0 ,\kappa} (\xi)$.
\begin{itemize}
\item Let $\xi\in \left[ 0,\frac{2\xi_0}{\kappa } \right] .$ Then we have
\begin{eqnarray} \label{estimatefxi0kappa1}
\left( 2\cot (2\xi ) -  \frac{\hat{f}_{\xi_0 ,\kappa}'(\xi)}{\hat{f}_{\xi_0 ,\kappa}(\xi)} \right)^2 + \frac{3\hat{f}_{\xi_0 ,\kappa}''(\xi)}{4\hat{f}_{\xi_0 ,\kappa}(\xi)} & = & \left( 2\cot (2\xi ) -  \kappa \cot (\kappa \xi) \right)^2 - \frac{3\kappa^2}{4} \\
& \leq & ( 2\xi + \kappa^2 \xi )^2 - \frac{3\kappa^2}{4} \leq -\frac{\kappa^2}{2} .\nonumber
\end{eqnarray}
\item Let $\xi\in \left[ \frac{2\xi_0}{\kappa } , \delta \right] .$ By a straightforward computation, one can see that
\begin{eqnarray}
& & \left( 2\cot (2\xi ) -  \frac{\hat{f}_{\xi_0 ,\kappa}'(\xi)}{\hat{f}_{\xi_0 ,\kappa}(\xi)} \right)^2 + \frac{3\hat{f}_{\xi_0 ,\kappa}''(\xi)}{4\hat{f}_{\xi_0 ,\kappa}(\xi)} \nonumber\\
& = & \left( 2\cot (2\xi ) -  \frac{\partial}{\partial \xi} \left( \log ( \sin (2\xi) \xi^{-\alpha } ) \right) \right)^2 + \frac{3  }{4\sin (2\xi) \xi^{-\alpha } } \frac{\partial^2}{\partial \xi^2} \left( \sin (2\xi) \xi^{-\alpha } \right) \label{estimatefxi0kappa2}\\
& = & \alpha^2 \xi^{-2} -3 - 3\alpha \cot (2\xi ) \xi^{-1} + \frac{3 \alpha (1+\alpha) }{4 } \xi^{-2} \nonumber \\
& = & -3 + \alpha \xi^{-1} \left(  \frac{3 +7\alpha }{4 } \xi^{-1} -3\cot (2\xi) \right) \leq -3 + \alpha \xi^{-1} \left(  \frac{3 +56\xi_0 }{4 } \xi^{-1} -\frac{3}{2\xi } + 3\xi \right) \leq -3 .\nonumber
\end{eqnarray}
\item Let $\xi\in \left[ \delta ,\frac{\pi}{2} \right] .$ Hence we have
\begin{eqnarray}
 \left| 2\cot (2\xi ) -  \frac{\hat{f}_{\xi_0 ,\kappa}'(\xi)}{\hat{f}_{\xi_0 ,\kappa}(\xi)} \right| & = & \left| \frac{-\alpha \xi^{-1-\alpha} \eta (\xi) + \eta'(\xi) \left( \xi ^{-\alpha } - \left(\frac{\tau}{2} \right)^{-\alpha} \right) }{\xi^{-\alpha} \eta (\xi) + \left( \frac{\tau}{2} \right)^{-\alpha} (1-\eta (\xi))} \right| \\
 & \leq & \alpha \tau^{\alpha} \delta^{-1-\alpha} + \sup_{\xi\in [0,\tau ]} |\eta'(\xi)| \left( \delta ^{-\alpha } - \tau^{-\alpha} \right) \nonumber ,
\end{eqnarray}
and
\begin{eqnarray}
  \left|4+ \frac{ \hat{f}_{\xi_0 ,\kappa}''(\xi)}{ \hat{f}_{\xi_0 ,\kappa}(\xi)} \right| & \leq & \left| \frac{4 \cot (2\xi) \left( \eta'(\xi) \left( \xi ^{-\alpha } - \left(\frac{\tau}{2} \right)^{-\alpha} \right) -\alpha \xi^{-1-\alpha} \eta (\xi)  \right) }{ \xi^{-\alpha} \eta (\xi) + \left( \frac{\tau}{2} \right)^{-\alpha} (1-\eta (\xi)) } \right| \nonumber \\
  &  & + \left| \frac{ \eta'' (\xi) \left( \xi ^{-\alpha } - \left(\frac{\tau}{2} \right)^{-\alpha} \right) -2\alpha \xi^{-1-\alpha} \eta'(\xi) + \alpha (1+\alpha) \xi^{-2-\alpha } \eta (\xi) }{ \xi^{-\alpha} \eta (\xi) + \left( \frac{\tau}{2} \right)^{-\alpha} (1-\eta (\xi)) } \right| \\
  & \leq & 4(1+\cot (2\delta )) \sup_{\xi\in [0,\tau ]} ( 1+ |\eta'(\xi)| + |\eta''(\xi)| ) \left( \left( \delta ^{-\alpha } - \tau^{-\alpha} \right) + \alpha \delta^{-2-\alpha} \right) .\nonumber
\end{eqnarray}
Note that $|\alpha| \leq 8\xi_0 $, $\forall \xi_0 \in (0,\delta ) $. One can easily to see that there exists $\delta' \in (0,\delta )$ such that for any $ \xi_0 \in (0,\delta' ) $ and $\xi\in \left[ \delta ,\frac{\pi}{2} \right] ,$
\begin{eqnarray}
 \left| 2\cot (2\xi ) -  \frac{\hat{f}_{\xi_0 ,\kappa}'(\xi)}{\hat{f}_{\xi_0 ,\kappa}(\xi)} \right| + \left|4+ \frac{ \hat{f}_{\xi_0 ,\kappa}''(\xi)}{ \hat{f}_{\xi_0 ,\kappa}(\xi)} \right| \leq \frac{1}{10} .
\end{eqnarray}
Now we choose $\xi_0 \in \left( 0, \frac{1}{100} \min \{ \delta , \delta' \} 
 \right) $. Then for any $\xi\in \left[ \delta ,\frac{\pi}{2} \right] ,$
 \begin{eqnarray} \label{estimatefxi0kappa3}
\left( 2\cot (2\xi ) -  \frac{\hat{f}_{\xi_0 ,\kappa}'(\xi)}{\hat{f}_{\xi_0 ,\kappa}(\xi)} \right)^2 + \frac{3\hat{f}_{\xi_0 ,\kappa}''(\xi)}{4\hat{f}_{\xi_0 ,\kappa}(\xi)} & \leq & -3 + \frac{1}{5} \leq -2.
\end{eqnarray}
\item Combining (\ref{estimatefxi0kappa1}), (\ref{estimatefxi0kappa2}) and (\ref{estimatefxi0kappa3}), one can see that
\begin{eqnarray} \label{estimatefxi0kappa4}
\left( 2\cot (2\xi ) -  \frac{\hat{f}_{\xi_0 ,\kappa}'(\xi)}{\hat{f}_{\xi_0 ,\kappa}(\xi)} \right)^2 + \frac{3\hat{f}_{\xi_0 ,\kappa}''(\xi)}{4\hat{f}_{\xi_0 ,\kappa}(\xi)} & \leq & -2,\;\;\; \forall \xi \in \left[ 0,\frac{\pi}{2} \right] .
\end{eqnarray}
\end{itemize}

Let $ \xi_0 $ be the constant we chosen in the above. Since 
\begin{eqnarray}
\lim_{\kappa\to\infty} \frac{4\xi_0}{\kappa} \cot \left( \frac{2\xi_0}{\kappa} \right) - 4\xi_0 \cot (2\xi_0 ) = 2-4\xi_0 \cot (2\xi_0 ) >0 ,
\end{eqnarray}
we see that $ \lim\limits_{\kappa\to\infty} \left( \frac{2\xi_0}{\kappa} \right)^\alpha =0 $. Hence for each $\kappa\geq 2$, one can find $\kappa'>\kappa$ such that
\begin{eqnarray}
\hat{f}_{\xi_0 ,\kappa} (\xi) = p \hat{f}_{\xi_0 ,\kappa'} (\xi),\;\; \forall \xi\in \left[ \tau ,\frac{\pi}{2} \right] .
\end{eqnarray}

For each $\kappa \geq 2$, set
\begin{equation}
\hat{f}_{ \kappa} (\xi) =\left\{
\begin{aligned}
 \hat{f}_{\xi_0 ,\kappa} (\xi) &, \textrm{ on $\left[ 0,\frac{\pi}{4 } \right] $,}\\
 p\hat{f}_{\xi_0 ,\kappa'} (\frac{\pi}{2} - \xi) &,  \textrm{ on $\left[ \frac{\pi}{4 } , \frac{\pi}{2} \right] $.}\\
\end{aligned}
\right.  
\end{equation} 
Then we have $\hat{f}_{ \kappa} \in C^{1,1} ([0,\frac{\pi}{2}]) \cap C^{\infty} ([0,\frac{2\xi_0}{\kappa}]) \cap C^{\infty} ([ \frac{2\xi_0}{\kappa} , \frac{\pi}{2} - \frac{2\xi_0}{\kappa'} ]) \cap C^{\infty} ([ \frac{\pi}{2} - \frac{2\xi_0}{\kappa'} ,\frac{\pi}{2} ]) $, and
\begin{eqnarray} \label{estimatefkappa1}
\left( 2\cot (2\xi ) -  \frac{\hat{f}'_{ \kappa}(\xi)}{\hat{f}_{ \kappa}(\xi)} \right)^2 + \frac{3\hat{f}_{ \kappa}''(\xi)}{4\hat{f}_{ \kappa}(\xi)} & \leq & -2,\;\;\; \forall \xi \in \left[ 0,\frac{\pi}{2} \right] .
\end{eqnarray}
By smoothing $\hat{f}_{ \kappa}$ around $\frac{2\xi_0}{\kappa} $ and $\frac{\pi}{2} - \frac{2\xi_0}{\kappa'} $, one can obtain the function $f_\kappa $. Without loss of generality, we can assume that $f_\kappa =\hat{f}_\kappa $ on $[0,\frac{ \xi_0}{\kappa}] \cup [ \frac{\pi}{2} - \frac{ \xi_0}{\kappa'} ,\frac{\pi}{2} ] $, and
\begin{eqnarray} \label{estimatefkappa2}
\left( 2\cot (2\xi ) -  \frac{f'_{ \kappa}(\xi)}{f_{ \kappa}(\xi)} \right)^2 + \frac{3f_{ \kappa}''(\xi)}{4f_{ \kappa}(\xi)} & \leq & -1,\;\;\; \forall \xi \in \left[ 0,\frac{\pi}{2} \right] .
\end{eqnarray}
Clearly, $f_\kappa $ satisfies (i), (ii) and (iii). It is sufficient to prove (iv) now.

Let $t_\kappa = \frac{1}{10} \inf_{\xi\in ( 0,\frac{\pi}{2})} \left( \left( \frac{\sin (2\xi )}{2f(\xi)} \right)^{-2} \right) $ and $\epsilon_\kappa = \frac{1}{10} \inf_{\xi\in ( 0,\frac{\pi}{2})} \min\left\{ \left( \frac{\sin (2\xi )}{2f(\xi)} \right)^{-2} ,- \frac{1}{4} \frac{f(\xi)}{f''(\xi)} \right\} $. Then for any $t\in (0,t_\kappa)$, (\ref{inequalityQft2}) implies that $Q_{f_\kappa ,t} \geq 0$, and hence
\begin{eqnarray}
  \Ric_{g_{f_\kappa ,t}} = Q_{f_\kappa ,t} + t^2 \left( \frac{\sin (2\xi )}{2f_\kappa (\xi)} \right)^2 d\alpha^2 - \frac{1}{4} \frac{f_\kappa''(\xi)}{f_\kappa (\xi)} \pi^*_{Hopf} h_{f_\kappa} \geq \epsilon_\kappa \left( t^2 d\alpha^2 + \pi_{Hopf}^* h_{f_\kappa} \right) ,
\end{eqnarray}
which proves (iv). 
\end{proof}

Our next lemma is a continuation of Lemma \ref{lemmakappariccipositive}. We construct suitable metrics $g_{\rho ,\phi,f} $ in the following lemma which look like $\mathbb{R}^2 \times C(\mathbb{S}^1_{\theta}) $ near the exceptional divisor $\blowup^{-1} (0) $.

\begin{lmm}\label{lemmaedgesingularities}
Let $\xi_0$, $\kappa$, $f_{\kappa}$ be the data in Lemma \ref{lemmakappariccipositive}. Then there exists a constant $\varepsilon_0 = \varepsilon_0 (\kappa) \in (0,\frac{1}{10}) $ such that for any $\mu\in (0,\frac{1}{10})$, we can find smooth functions $\rho_{\kappa ,\mu}$ and $\phi_{\kappa ,\mu}$ satisfying the following properties.
\begin{enumerate}[\rm (i).]
    \item $\Ric_{g_{\rho_{\kappa ,\mu} ,\phi_{\kappa ,\mu} , f_\kappa}} \geq 0$,
    \item For any $r\in \left( 0,\frac{1}{10\kappa } \right) $, $\phi_{\kappa ,\mu} (r) =1$,
    \item For any $\varepsilon \in (0,\varepsilon_0 ) $ and $r\in \left( 0,\frac{1}{10\kappa } \right) $, $\left|\frac{\rho_{\kappa ,\mu} (r)}{n} - \varepsilon \sin r \right| \leq \mu $, $ \frac{\rho_{\kappa ,\mu} (r)}{n} \leq \sin r $ $ \frac{\rho'_{\kappa ,\mu } (r) \sin (\kappa r) }{ \kappa \rho_{\kappa ,\mu } (r) \cos (\kappa r) } \in [1-\mu ,1+\mu] $, and $ \frac{-\rho''_{\kappa ,\mu } (r) }{ \kappa^2 \rho_{\kappa ,\mu } (r) } \geq 1-\mu $,
    \item There are constants $c_1, c_2 ,c_3 >0$ and $R>0$ such that for any $r\geq R$, $\rho_{\kappa ,\mu} (r) = c_1 (r+ c_3 ) $, and $\phi_{\kappa ,\mu} (r) = c_2 (r+ c_3 ) $.
\end{enumerate}
\end{lmm}

\begin{proof}
Let $\hat{\mu} \in (0,\frac{1}{100\kappa} ) $ be a constant to be specified later. Set
\begin{equation}
\hat{\rho}_{ \kappa,\mu } (r) =\left\{
\begin{aligned}
 n \frac{\sin (\kappa r)}{\kappa}\;\;\;\;\;\;\;\;\;\;\;\;\; &, \textrm{ on $\;\;\left[ 0,\hat{\mu} \right] $,}\\
n \left( \sin (\kappa \hat{\mu}) \right)^{\frac{\mu}{2}} \left( \sin (\kappa r ) \right)^{1-\frac{\mu}{2}} &,  \textrm{ on $\left[ \hat{\mu} , \frac{1}{5\kappa } \right] $.}\\
\end{aligned}
\right.  
\end{equation}
It is easy to see that $ \hat{\rho}_{ \kappa,\mu } \in C^{0,1} ( [ 0,\frac{1}{5\kappa} ] ) \cap C^{\infty } ( [ 0,\hat{\mu} ] ) \cap C^{\infty } ( [ \hat{\mu} , \frac{1}{5\kappa} ] ) $.

By a straightforward calculation, we arrive at
\begin{equation}
\frac{\hat{\rho}'_{ \kappa,\mu } (r)}{\hat{\rho}_{ \kappa,\mu } (r)} =\left\{
\begin{aligned}
 \kappa \cot (\kappa r) \;\;\;\;\;\; &, \textrm{ on $\;\;\left[ 0,\hat{\mu} \right] $,}\\
\left( 1-\frac{\mu}{2} \right) \kappa  \cot (\kappa r) &,  \textrm{ on $\left[ \hat{\mu} , \frac{1}{5\kappa } \right] $,}\\
\end{aligned}
\right.  
\end{equation}
and
\begin{equation}
\frac{\hat{\rho}''_{ \kappa,\mu } (r)}{\hat{\rho}_{ \kappa,\mu } (r)} =\left\{
\begin{aligned}
 -\kappa^2 \;\;\;\;\;\;\;\;\;\;\;\;\;\;\;\;\;\;\;\;\; &, \textrm{ on $\;\;\left[ 0,\hat{\mu} \right] $,}\\
-\left( 1-\frac{\mu}{2} \right) \kappa^2 \left( 1 + \frac{\mu}{2}  \cot^2 (\kappa r) \right) &,  \textrm{ on $\left[ \hat{\mu} , \frac{1}{5\kappa } \right] $.}\\
\end{aligned}
\right.  
\end{equation}
Since $\frac{\hat{\rho}'_{ \kappa,\mu } (r) \tan (\kappa r ) }{\kappa \hat{\rho}_{ \kappa,\mu } (r)} \in [1-\frac{ \mu}{2 } ,1 + \frac{ \mu}{2 } ] $, $- \frac{\hat{\rho}''_{ \kappa,\mu } (r) }{\kappa^2 \hat{\rho}_{ \kappa,\mu } (r)} \geq 1-\frac{1}{2\mu } $ and $ \hat{\rho}'_{ \kappa,\mu } (\hat{\mu}_+) < \hat{\rho}'_{ \kappa,\mu } (\hat{\mu}_- ) $, one can obtain a function $ \tilde{\rho}_{ \kappa,\mu } \in C^{\infty} ( [ 0,\frac{1}{5\kappa} ] ) $ by smoothing $ \hat{\rho}_{ \kappa,\mu } $ near $\hat{\mu}$ such that $\tilde{\rho}_{ \kappa,\mu } = \hat{\rho}_{ \kappa,\mu } $ on $[ \frac{1}{10\kappa},\frac{1}{5\kappa} ]$, and $\tilde{\rho}_{ \kappa,\mu }$ satisfies condition (iii) except the first inequality.

Let
\begin{equation}
\hat{\phi}_{ \kappa,\mu } (r) =\left\{
\begin{aligned}
 1 \;\;\;\;\;\;\;\;\;\;\;\;\;\;\;\;\;\;\;\;\;\;\;\;\;\;\;\;\;\;\;\;\;\; &, \textrm{ on $\;\;\left[ 0,\frac{1}{8\kappa } \right] $,}\\
 1 + (\epsilon_\kappa \hat{\mu} )^{20} \left( r - \frac{1}{8\kappa} \right)^4 \;\;\;\;\;\;\;\;\;\;\;\;\;\;\;\;\;\; &, \textrm{ on $\;\;\left[ \frac{1}{8\kappa } , \frac{3}{20\kappa } \right] $,}\\
1 + (\epsilon_\kappa \hat{\mu} )^{20} \left( \frac{1}{40\kappa } \right)^4 + 4 (\epsilon_\kappa \hat{\mu} )^{20} \left( \frac{1}{40\kappa } \right)^3 \left( r - \frac{1}{8\kappa} \right) &,  \textrm{ on $\left[ \frac{3}{20\kappa } , \infty \right) $.}\\
\end{aligned}
\right.  
\end{equation}
where $ \epsilon_\kappa \in (0,\frac{1}{10}) $ is the constant in Lemma \ref{lemmakappariccipositive}. It is easy to check that $\hat{\phi}_{ \kappa,\mu } \in C^{1,1} ([0,\infty )) $, and $ |\hat{\phi}_{ \kappa,\mu } | + |\hat{\phi}_{ \kappa,\mu } | \leq (\epsilon_\kappa \hat{\mu} )^{19} $. By smoothing $ \hat{\phi}_{ \kappa,\mu } $ near $\frac{1}{8\kappa}$ and $\frac{3}{20\kappa}$, we can obtain a function $\tilde{\phi}_{\kappa ,\mu} \in C^\infty ([0,\infty)) $ such that $\hat{\phi}_{ \kappa,\mu } = \tilde{\phi}_{\kappa ,\mu} $ on $[0,\frac{1}{10\kappa}] \cup [ \frac{1}{5\kappa } ,\infty ) $, and $ |\tilde{\phi}'_{ \kappa,\mu } | + |\tilde{\phi}'_{ \kappa,\mu } | \leq (\epsilon_\kappa \hat{\mu} )^{18} $.

Choose $\hat{\mu} >0$ such that $(\sin(\kappa \hat{\mu}) )^\mu \leq \left( \frac{t_\kappa}{100n\kappa } \right)^2 $, where $ t_\kappa \in (0,\frac{1}{10}) $ is the constant in Lemma \ref{lemmakappariccipositive}. Let $\varepsilon_0 = \frac{t_\kappa}{100n\kappa } $ and $\varepsilon = (\sin(\kappa \hat{\mu}) )^\frac{\mu}{2} $. Then $\tilde{\rho}_{ \kappa,\mu }$ satisfies the first inequality in condition (iii), $ \frac{\tilde{\rho}^2_{\kappa ,\mu}}{ \tilde{\phi}^2_{\kappa ,\mu} } < t_\kappa $ on $[0, \frac{1}{5\kappa } ] $, and Lemma \ref{lemmakappariccipositive} implies that for any $r\in [0 , \frac{1}{5\kappa } ]$
\begin{eqnarray}
    \Ric_{g_{\tilde{\rho}_{\kappa ,\mu} , \tilde{\phi}_{\kappa ,\mu} ,f_\kappa }} & = & -\left( \frac{\tilde{\rho}''_{\kappa ,\mu}}{\tilde{\rho}_{\kappa ,\mu}} + 2 \frac{\tilde{\phi}''_{\kappa ,\mu}}{\tilde{\phi}_{\kappa ,\mu}} \right) dr^2 + \Ric_{g_{f_\kappa , \frac{\tilde{\rho}^2_{\kappa ,\mu}}{ \tilde{\phi}^2_{\kappa ,\mu} } }} - \left( 2\frac{\tilde{\rho}_{\kappa ,\mu} \tilde{\rho}'_{\kappa ,\mu} \tilde{\phi}'_{\kappa ,\mu} }{\tilde{\phi}_{\kappa ,\mu}} + \tilde{\rho}_{\kappa ,\mu} \tilde{\rho}''_{\kappa ,\mu} \right) d\alpha^2 \nonumber\\
    & & - \left( \frac{\tilde{\rho}_{\kappa ,\mu} \tilde{\rho}'_{\kappa ,\mu} \tilde{\phi}'_{\kappa ,\mu} }{\tilde{\phi}_{\kappa ,\mu}} + \tilde{\phi}_{\kappa ,\mu} \tilde{\phi}''_{\kappa ,\mu} + (\tilde{\phi}'_{\kappa ,\mu})^2 \right) \pi_{Hopf}^* h_{f_\kappa } \nonumber\\
    & \geq & -\left( \frac{\tilde{\rho}''_{\kappa ,\mu}}{\tilde{\rho}_{\kappa ,\mu}} + 2 \frac{\tilde{\phi}''_{\kappa ,\mu}}{\tilde{\phi}_{\kappa ,\mu}} \right) dr^2 + \left( \epsilon_\kappa \frac{\tilde{\rho}^4_{\kappa ,\mu}}{ \tilde{\phi}^4_{\kappa ,\mu} } - 2\frac{\tilde{\rho}_{\kappa ,\mu} \tilde{\rho}'_{\kappa ,\mu} \tilde{\phi}'_{\kappa ,\mu} }{\tilde{\phi}_{\kappa ,\mu}} - \tilde{\rho}_{\kappa ,\mu} \tilde{\rho}''_{\kappa ,\mu} \right) d\alpha^2 \\
    & & + \left( \epsilon_\kappa - \frac{\tilde{\rho}_{\kappa ,\mu} \tilde{\rho}'_{\kappa ,\mu} \tilde{\phi}'_{\kappa ,\mu} }{\tilde{\phi}_{\kappa ,\mu}} - \tilde{\phi}_{\kappa ,\mu} \tilde{\phi}''_{\kappa ,\mu} - (\tilde{\phi}'_{\kappa ,\mu})^2 \right) \pi_{Hopf}^* h_{f_\kappa } .\nonumber 
\end{eqnarray}

Since $\tilde{\phi}_{\kappa ,\mu} =1 $ on $ [0, \frac{1}{10\kappa } ]$, one can see that for any $r\in [0, \frac{1}{10\kappa } ] $,
\begin{eqnarray}
 \Ric_{g_{\tilde{\rho}_{\kappa ,\mu} , \tilde{\phi}_{\kappa ,\mu} ,f_\kappa }} & \geq & \left( 1-\mu \right) dr^2 + \left( \epsilon_\kappa \frac{\tilde{\rho}^4_{\kappa ,\mu}}{ \tilde{\phi}^4_{\kappa ,\mu} } + (1-\mu )\tilde{\rho}^2_{\kappa ,\mu} \right) d\alpha^2 + \epsilon_\kappa \pi_{Hopf}^* h_{f_\kappa } \geq 0 .
\end{eqnarray}
and for any $r\in [\frac{1}{10\kappa } , \frac{1}{5\kappa } ]$
\begin{eqnarray}
 \Ric_{g_{\tilde{\rho}_{\kappa ,\mu} , \tilde{\phi}_{\kappa ,\mu} ,f_\kappa }} & \geq & \left( 1-\mu - 2(\epsilon_\kappa \hat{\mu} )^{15} \right) dr^2 + \left( \epsilon_\kappa \left(\frac{\hat{\mu}}{100}\right)^4 - 2n^2 \kappa^{ \mu} \hat{\mu}^{17+ \mu} \epsilon^{18}_\kappa \right) d\alpha^2 \\
 & & + \left( \epsilon_\kappa - n^2 \kappa^{ \mu} \hat{\mu}^{17+ \mu} \epsilon^{18}_\kappa - 2 \epsilon^{18}_\kappa \hat{\mu}^{18} - \epsilon^{36}_\kappa \hat{\mu}^{36} \right) \pi_{Hopf}^* h_{f_\kappa } \geq 0.\nonumber
\end{eqnarray}

Write $c_2 = 4 (\epsilon_\kappa \hat{\mu} )^{20} \left( \frac{1}{40\kappa } \right)^3 $, $c_3 = \frac{(40\kappa)^3  }{4 (\epsilon_\kappa \hat{\mu} )^{20} } - \frac{19}{160\kappa } $. By definition, for any $r\geq \frac{1}{5\kappa}$, we have $\tilde{\phi}_{\kappa ,\mu} (r) =c_2 (r+c_3) $. Since $\frac{1}{c_3 +\frac{1}{5\kappa } } \leq (1-\mu) \kappa \cot (\frac{1}{5}) = \frac{\tilde{\rho}'_{ \kappa,\mu } (\frac{1}{5\kappa})}{\tilde{\rho}_{ \kappa,\mu } (\frac{1}{5\kappa})} $, we can find a function $\rho_{\kappa ,\mu } \in C^{\infty} ([0,\infty)) $ satisfies that $\rho_{\kappa ,\mu} = \tilde{\rho}_{\kappa ,\mu} $ on $[0,\frac{1}{5\kappa}]$, $\rho''_{\kappa ,\mu} \leq 0 $, and $\rho_{\kappa ,\mu} (r) = c_1 (r+c_3 ) $ on $[ \frac{1}{4\kappa} ,\infty)$, where $c_1 >0$ is a constant such that $\frac{\tilde{\rho}_{ \kappa,\mu } (\frac{1}{5\kappa})}{c_3 +\frac{1}{5\kappa } } \leq c_1 \leq \frac{2\tilde{\rho}_{ \kappa,\mu } (\frac{1}{5\kappa})}{c_3 +\frac{1}{5\kappa } } $.

Let $\phi_{\kappa ,\mu} = \tilde{\phi}_{\kappa ,\mu} $. Then the functions $\rho_{\kappa ,\mu} , \phi_{\kappa ,\mu}$ satisfy the conditions (ii), (iii), (iv). By applying Lemma \ref{lemmakappariccipositive} again, we see that $\Ric_{g_{\rho_{\kappa ,\mu} ,\phi_{\kappa ,\mu} , f_\kappa}} \geq 0$, and (i) is proved. This completes the proof.
\end{proof}

\begin{rmk}
Actually, we only need to use Lemma \ref{lemmakappariccipositive} and Lemma \ref{lemmaedgesingularities} in the case $\kappa =2$.
\end{rmk}

\subsection{Adding conical singularities}
In the last step, we have constructed Riemannian metrics $g_{\rho_{\kappa ,\mu} ,\phi_{\kappa ,\mu} , f_\kappa}$ on the orbifold $(\mathcal{O}(-n) /\Gamma_{p,0,1} ) $ with non-negative Ricci curvature and Euclidean volume growth. But the tangent cone at the singular point $\{(0,0,[0,1]) \}$ is isometric to $C(\mathbb{S}^3 /\Gamma_{p,-n,1})$ for some small constant $\delta >0$. Now we will modify the metric near the singular point $\{(0,0,[0,1]) \}$ such that the tangent cone at $\{(0,0,[0,1]) \}$ is isometric to $C((N,g_N))$ for some manifold $(N,g_N)$ with $\Ric_{g_N} >1 $ and small volume.

For calculate the curvature near the point $ (0,0,[0,1]) $, we need to use another coordinate. Now we summarize the results obtained in the last two steps into the following proposition.

\begin{prop}\label{propositionlocalproductstructure}
Let $(n,p)$ be a coprime pair of integers, and $n>p\geq 1$. Then there exists a constant $r_0 = r_0 (n,p) >0$ satisfies the following properties. For any $\mu\in (0,\frac{1}{100})$, there exist
\begin{itemize}
    \item A Riemannian metric $g_{n,p,\mu }$ on $(\mathcal{O}(-n) /\Gamma_{p,0,1} ) \sq ( \proj^{-1} ([1,0]) \cup \{(0,0,[0,1]) \} ) $,
    \item A Riemannian metric $h_{n,p,\mu}$ on $\mathbb{S}^3 /\Gamma_{n,1,p} $,
    \item A positive constant $R_\mu >0$,
    \item A non-negative function $\rho_\mu$ on $[0,r_0]$,
    \item A compact subset $K_\mu \subset (\mathcal{O}(-n) /\Gamma_{p,0,1} ) $ containing the singular point $ (0,0,[0,1]) $,
\end{itemize}
such that
\begin{itemize}
    \item The extension of the metric $g_{n,p,\mu }$ on $(\mathcal{O}(-n) /\Gamma_{p,0,1} ) \sq \{ (0,0,[0,1]) \}$ gives a smooth Riemannian manifold structure on $(\mathcal{O}(-n) /\Gamma_{p,0,1} ) \sq \{ (0,0,[0,1]) \}$ with $\Ric_{g_{n,p,\mu}} \geq 0$,
    \item The Riemannian manifold $((\mathcal{O}(-n) /\Gamma_{p,0,1} ) \sq K_\mu ,g_{n,p,\mu} )$ is isometric to the annulus $A_{R_\mu ,\infty } (o) $ in the metric cone $ (C(\mathbb{S}^3 /\Gamma_{n,1,p} ) ,dr^2 +r^2 h_{n,p,\mu} ,o ) $,
    \item The function $\frac{\rho_{ \mu} (r)}{n} \leq \min \{ \sin ( r) , \mu + \mu \sin ( r) \} $, $ \left|\frac{\rho'_{ \mu } (r) \sin (2 r) }{ 2 \rho_{ \mu } (r) \cos (2 r) } -1 \right| \leq \mu $, and $ \frac{-\rho''_{ \mu } (r) }{ 4 \rho_{ \mu } (r) } \geq 1-\mu $,
    \item The metric ball $B_{r_0} ((0,0,[0,1])) \subset \mathcal{O}(-n) /\Gamma_{p,0,1} $ is isometric to the metric ball $B_{r_0} ((0,0)) \subset ( W_{\frac{\sin (2r)}{2}} \times W_{\frac{\rho_\mu (r)}{n}}  ) / G_{n,p,\mu } $ for some finite group $G_{n,p,\mu } < \mathbb{S}^1 \times \mathbb{S}^1 \leq \Aut (W_{\frac{\sin (2r)}{2}} \times W_{\frac{\rho_\mu (r)}{n}} ) $,
\end{itemize}
where $\proj^{-1} ([1,0]) \subset \mathcal{O}(-n) $ is the fiber over $[1,0] \in \mathbb{C}P^1 $, and $W_\varphi $ is the $\mathbb{S}^1$-invariant Riemannian manifold given by the rotationally symmetric warped product structure $( [0,r_0 ] \times \mathbb{S}^1 ,dr^2 + \varphi(r)^2 d\theta^2 ) $.
\end{prop}

\begin{proof}
By letting $\kappa =2$ and $\varepsilon \leq \mu $ in Lemma \ref{lemmakappariccipositive} and Lemma \ref{lemmaedgesingularities}, one can obtain a Riemannian metric $\tilde{g}_{n,p,\mu}$ on $(\mathcal{O}(-n) /\Gamma_{p,0,1} ) \sq ( \proj^{-1} ([1,0]) \cup \{(0,0,[0,1]) \} ) $, a Riemannian metric $h_{n,p,\mu}$ on $\mathbb{S}^3 /\Gamma_{n,1,p} $, a positive constant $R_\mu >0$, a non-negative function $\rho_\mu$ on $[0,r_0]$, and a compact subset $K_\mu \subset (\mathcal{O}(-n) /\Gamma_{p,0,1} ) $ containing the singular point $ (0,0,[0,1]) $ satisfying the properties except the last one. So we only need to modify the metric $\tilde{g}_{n,p,\mu}$ near the point $(0,0,[0,1])$ to satisfy the last property.

Now we consider the coordinate:
\begin{equation*}
\hat{F}:\left\{
\begin{aligned}
\relax \left( 0,\frac{\pi}{2} \right) \times ( 0,2\pi )\times (0,2\pi ) \;\;\;\;& \longrightarrow \;\;\;\;\;\;\;\;\;\;\;\;\;\;\;\;\;\;\;\;\;\;\;\; \mathbb{S}^3 /\mathbb{Z}_n ,\\
(\hat{\xi} ,\hat{\alpha} ,\hat{\beta})  \;\;\;\;\;\;\;\;\;\;\;\;\;\;\;\; &\longmapsto \;\;\;\; \left( \sin\xi e^{\frac{\hat{\alpha}}{n} \sqrt{-1}} , \cos \xi e^{(\frac{\hat{\alpha}}{n} +\hat{\beta} )\sqrt{-1}} \right) , \\
\end{aligned}
\right.  
\end{equation*}
and the Hopf fibration $\pi_{Hopf} ( z_1 , z_2 ) = \left( z_1 \bar{z}_2 , \frac{1}{2} (|z_1|^2 -|z_2|^2 ) \right) $. Since $\pi_{Hopf} ( e^{\zeta \sqrt{-1}} z_1 , e^{\zeta \sqrt{-1}} z_2 ) = \pi_{Hopf} ( z_1 , z_2 ) $, $\forall \zeta\in\mathbb{R} $, we see that $ \pi_{Hopf} : \mathbb{S}^3 /\mathbb{Z}_n \to \mathbb{S}^2 $ is well-defined. It is easy to see that the basic horizontal lift of $D_{\pi_{Hopf}} ( \frac{\partial }{\partial \hat{\xi} } )$ and $D_{\pi_{Hopf}} ( \frac{\partial }{\partial \hat{\beta} } )$ are $\frac{\partial }{\partial \hat{\xi} }$ and $\frac{\partial }{\partial \hat{\beta} } + n\sin^2 (\hat{\xi}) \frac{\partial }{\partial \hat{\alpha} } $, respectively. Let $\xi_0 $, $f_{2}$ and $ \phi_{2,\mu} $ be the data in Lemma \ref{lemmakappariccipositive} and Lemma \ref{lemmaedgesingularities}. Then 
$$ \left\{ \frac{\partial}{\partial r} ,\;\;\; \frac{n}{\rho_{\mu} (r) } \frac{\partial}{\partial \hat{\alpha}} ,\;\;\; \frac{1}{\phi_{2,\mu} (r) } \frac{\partial}{\partial \hat{\xi}} ,\;\;\; \frac{1}{\phi_{2,\mu} (r) f_2 (\hat{\xi})} \left( \frac{\partial }{\partial \hat{\beta} } + n\sin^2 (\hat{\xi}) \frac{\partial }{\partial \hat{\alpha} } \right) \right\} $$
gives an orthonormal basis corresponding to $\tilde{g}_{n,p,\mu}$ on this coordinate. Moreover, for any $ \hat{\xi} \in [0,\frac{\xi_0}{2}] $, we have $\phi_{2,\mu} (\hat{\xi}) = 1 $ and $f_{2} (\hat{\xi}) = \frac{\sin (2\hat{\xi})}{2} $. Hence when $r,\hat{\xi} \in [0,\frac{\xi_0}{2}] $, the orthonormal basis above corresponding to $\tilde{g}_{n,p,\mu}$ becomes
$$ \left\{ \frac{\partial}{\partial r} ,\;\;\; \frac{n}{\rho_{\mu} (r) } \frac{\partial}{\partial \hat{\alpha}} ,\;\;\;  \frac{\partial}{\partial \hat{\xi}} ,\;\;\; \frac{2}{ \sin (2\hat{\xi})} \left( \frac{\partial }{\partial \hat{\beta} } + n\sin^2 (\hat{\xi}) \frac{\partial }{\partial \hat{\alpha} } \right) \right\} .$$

Now we glue $ ( W_{\frac{\sin (2r)}{2}} \times W_{\frac{\rho_\mu (r)}{n}}  ) $ to $\mathcal{O}(-n)$ near the point $(0,0,[0,1])$.

For any $\sigma \in \left( 0,\frac{\xi_0}{100} \right) $, set $\eta_\sigma (t)$ be the cut-off function on $[0,1]$ such that $\eta_\sigma (t) =1 $ on $[0,\sigma ]$, $\eta_\sigma (t) =0 $ on $[2\sigma ,1]$, and $| \eta'_\sigma | \leq \frac{2}{\sigma} $. Then we consider the four vector fields:
$$ X_1 = \frac{\partial}{\partial r} ,\;\;\; X_2 = \frac{n}{\rho_{\mu} (r) } \frac{\partial}{\partial \hat{\alpha}} ,\;\;\; X_3 = \frac{\partial}{\partial \hat{\xi}} ,\;\;\; X_4 = \frac{2}{ \sin (2\hat{\xi})} \left( \frac{\partial }{\partial \hat{\beta} } + n(1-\eta_{\sigma_1} (r) \eta_{\sigma_2} (\hat{\xi}) )\sin^2 (\hat{\xi}) \frac{\partial }{\partial \hat{\alpha} } \right) ,$$
where $\sigma_1 ,\sigma_2 \in \left( 0,\frac{\xi_0}{100} \right) $ are positive constants to be determined. One can see that there exists a unique smooth Riemannian metric $\tilde{g}_{\mu ,\sigma_1 ,\sigma_2}$ on $V_{\xi_0} 
=\left\{ (r,\hat{\xi} ,\hat{\alpha} ,\hat{\beta}) :\;\; r,\hat{\xi} \in [0,\frac{\xi_0}{2}] \right\} $ such that $\tilde{g}_{\mu ,\sigma_1 ,\sigma_2} (X_i ,X_j ) =\delta_{i,j} $, where $\delta_{i,j}$ is the Kronecker symbol.

For abbreviation, we write $\psi (r,\hat{\xi})$ instead of $n( \eta_{\sigma_1} (r) \eta_{\sigma_2} (\hat{\xi}) -1 )\sin^2 (\hat{\xi})$, and we will subscript the variable that we differentiated with respect to. For example, we abbreviate $\frac{\partial^2 \psi }{\partial r \partial \hat{\xi} }$ to $\psi_{r\hat{\xi}}$ or $\psi_{ \hat{\xi}r}$. By a straightforward calculation similar to the computation below Fact \ref{factwarpedbergersphere}, we can conclude that
\begin{eqnarray}
\Ric_{\tilde{g}_{\mu ,\sigma_1 ,\sigma_2}} (X_1) & = & \left( -\frac{\rho''_\mu}{\rho_\mu} - \frac{2\rho_\mu^2 \psi_r^2 }{n^2 \sin^2 (2\hat{\xi})} \right) X_1 -\frac{2\rho_\mu^2 \psi_r \psi_{\hat{\xi}} }{n^2 \sin^2 (2\hat{\xi})} X_2 , \\
\Ric_{\tilde{g}_{\mu ,\sigma_1 ,\sigma_2}} (X_2) & = & \left(4 - \frac{2\rho_\mu^2 \psi_{\hat{\xi}}^2 }{n^2 \sin^2 (2\hat{\xi})} \right) X_2 -\frac{2\rho_\mu^2 \psi_r \psi_{\hat{\xi}} }{n^2 \sin^2 (2\hat{\xi})}  X_1 ,\\
\Ric_{\tilde{g}_{\mu ,\sigma_1 ,\sigma_2}} (X_3) & = & \left( -\frac{\rho''_\mu}{\rho_\mu} + \frac{2\rho_\mu^2 \psi_r^2 }{n^2 \sin^2 (2\hat{\xi})}  + \frac{2\rho_\mu^2 \psi_{\hat{\xi}}^2 }{n^2 \sin^2 (2\hat{\xi})} \right) X_3 \\
& & - \left( \frac{3\rho'_\mu \psi_r }{ n\sin (2\hat{\xi})} + \frac{\rho_\mu \psi_r^2}{n\sin (2\hat{\xi})} + \frac{\rho_\mu \psi_{\hat{\xi}}^2 }{n\sin (2\hat{\xi})} - \frac{2\rho_\mu  \psi_{\hat{\xi}} \cos (2\hat{\xi}) }{n\sin^2  (2\hat{\xi})}  \right) X_4 ,\nonumber\\
\Ric_{\tilde{g}_{\mu ,\sigma_1 ,\sigma_2}} (X_4) & = & \left( 4 - \frac{2\rho_\mu^2 \psi_r^2 }{n^2 \sin^2 (2\hat{\xi})}  - \frac{2\rho_\mu^2 \psi_{\hat{\xi}}^2 }{n^2 \sin^2 (2\hat{\xi})}  \right) X_4 \\
& & - \left( \frac{3\rho'_\mu \psi_r }{ n\sin (2\hat{\xi})} + \frac{\rho_\mu \psi_r^2}{n\sin (2\hat{\xi})} + \frac{\rho_\mu \psi_{\hat{\xi}}^2 }{n\sin (2\hat{\xi})} - \frac{2\rho_\mu  \psi_{\hat{\xi}} \cos (2\hat{\xi}) }{n\sin^2  (2\hat{\xi})}  \right) X_3 . \nonumber
\end{eqnarray}
By definition, we have $|\rho_\mu|\leq 2n\mu$, $|\rho'_{\mu}|\leq n$, $-\rho''_{\mu} \geq 4(1-\mu) \rho_\mu $ and $\left| \frac{\psi_r }{\sin (2\hat{\xi})} \right| \leq \frac{ 2n \sigma_2 }{\sigma_1}  $. Fix $\sigma_1 = \frac{\xi_0}{200} $ and $\sigma_2 = \frac{\sigma_1 }{200n^2} $. Then $\left| \frac{3\rho'_\mu \psi_r }{ n\sin (2\hat{\xi})} \right| \leq \frac{1}{100} $. By choosing $\mu$ small enough, one can get
\begin{eqnarray}
    \sum_{i=1}^2 \left( \left| \Ric_{\tilde{g}_{\mu ,\sigma_1 ,\sigma_2}} (X_{2i-1}) + \frac{\rho''_\mu}{\rho_\mu} X_{2i-1} \right| + \left| \Ric_{\tilde{g}_{\mu ,\sigma_1 ,\sigma_2}} (X_{2i}) - 4 X_{2i} \right| \right) \leq \frac{1}{10} ,
\end{eqnarray}
and hence $ \Ric_{\tilde{g}_{\mu ,\sigma_1 ,\sigma_2}} \geq 0 $. Let
\begin{equation*}
g_{n,p,\mu } :\left\{
\begin{aligned}
\tilde{g}_{\mu ,\sigma_1 ,\sigma_2} &,\;\;\;\;\;\;\;\;\;\; \textrm{ on $\;\;\;\;\;\;\;\;\;\; V_{\xi_0} $} ,\\
\tilde{g}_{n,p,\mu} &,\;\;\;\;\;\;\;\;\;\; \textrm{ on $(\mathcal{O}(-n) /\Gamma_{p,0,1} ) \sq V_{\xi_0} $} , \\
\end{aligned}
\right.  
\end{equation*}
and $r_0 = \frac{\xi_0}{1000} $. Note that $\Gamma_{p,0,1}$ is a finite subgroup of the automorphism group generated by $\frac{\partial}{\partial \hat{\alpha}}$ and $\frac{\partial}{\partial \hat{\beta}} $, which gives the last condition we need. This proves the proposition.
\end{proof}

Then we focus on the Riemannian metric on $ \mathcal{W}_\mu = W_{\frac{\sin (2r)}{2}} \times W_{\frac{\rho_\mu (r)}{n}}$. By the warped product structures of $W_{\frac{\sin (2r)}{2}} $ and $ W_{\frac{\rho_\mu (r)}{n}}$, one can get a natural coordinate $(r_1 ,\theta_1 ,r_2 ,\theta_2 )$, and the metric
$$ g_{\mathcal{W}} = g_{W_{\frac{\sin (2r)}{2}}} + g_{W_{\frac{\rho_\mu (r)}{n}}} = dr_1^2 + \frac{\sin^2 (2r_1)}{4} d\theta_1^2 + dr_2^2 + \frac{\rho^2_\mu (r_2)}{n^2} d\theta_2^2 .$$

Write $\gamma =\sqrt{r^2_1 + r^2_2 } \in [0,r_0] $ and $\theta = \arcsin \left( \frac{r_2}{\gamma} \right) \in [0,\frac{\pi}{2}] $. Then $r_1 = \gamma \cos \theta $, $r_2 = \gamma \sin \theta $, and
$$ g_{\mathcal{W}} = d\gamma^2 + \gamma^2 \left( d\theta^2 + \frac{\sin^2 (2\gamma \cos \theta)}{4 \gamma^2 } d\theta_1^2 + \frac{\rho^2_\mu (\gamma \sin \theta)}{n^2 \gamma^2 } d\theta_2^2 \right) .$$

Since we need to modify the metric under the restriction that it remains invariant under the $\mathbb{T}^2 = (\mathbb{S}^1 \times \mathbb{S}^1 )$-action, we are now constructing a family of metrics that satisfy this condition. By using the coordinate $(\gamma ,\theta ,\theta_1 ,\theta_2)$ constructed above, we obtain the following fact.

\begin{fact}\label{factT2invarantmetric}
Let $ \Phi ,\Psi , \Upsilon$ be non-negative smooth functions on $ (0,r_0] \times [0,\frac{\pi}{2} ] $ such that:
\begin{itemize}
    \item $\Phi >0$ on $ (0,r_0] \times [0,\frac{\pi}{2} ] $, $\frac{\partial^{(odd)}\Phi (\gamma ,\theta ) }{\partial \theta^{(odd)} } (\gamma ,0) = \frac{\partial^{(odd)}\Phi (\gamma ,\theta ) }{\partial \theta^{(odd)} } (\gamma ,\frac{\pi}{2} ) = 0 $,
    \item $\Psi >0$ on $ (0,r_0] \times [0,\frac{\pi}{2} ) $, $\frac{\partial^{(odd)}\Psi (\gamma ,\theta ) }{\partial \theta^{(odd)} } (\gamma ,0) = \frac{\partial^{(even)}\Psi (\gamma ,\theta ) }{\partial \theta^{(even)} } (\gamma ,\frac{\pi}{2} ) = \Psi (\gamma ,\frac{\pi}{2} ) = 0 $,\\ $\frac{\partial \Psi (\gamma ,\theta ) }{\partial \theta } (\gamma ,\frac{\pi}{2} ) = -\Phi (\gamma ,\frac{\pi}{2} ) $,
    \item $\Upsilon >0$ on $ (0,r_0] \times (0,\frac{\pi}{2} ] $, $\frac{\partial^{(odd)}\Upsilon (\gamma ,\theta ) }{\partial \theta^{(odd)} } (\gamma ,\frac{\pi}{2}) = \frac{\partial^{(even)}\Upsilon (\gamma ,\theta ) }{\partial \theta^{(even)} } (\gamma ,0 ) = \Upsilon (\gamma ,0 ) = 0 $,\\ $\frac{\partial \Psi (\gamma ,\theta ) }{\partial \theta } (\gamma ,0 ) = \Phi (\gamma ,0 ) $.
\end{itemize}
Then we have the following metrics.
\begin{itemize}
    \item For any $\gamma \in (0,r_0 ] $, the metric on $ [0,\frac{\pi}{2} ] \times [0,\frac{\pi}{2} ] \times [0,\frac{\pi}{2} ] $,
    $$ g_{\Phi_{\gamma } , \Psi_\gamma , \Upsilon_\gamma } = \Phi (\gamma ,\theta) d\theta^2 + \Psi (\gamma ,\theta) d\theta_1^2 + \Upsilon (\gamma ,\theta) d\theta_2^2  $$
    gives a $\mathbb{T}^2$-invariant smooth Riemannian metric on $\mathbb{S}^3$, where the $\mathbb{T}^2$-action is induced by $\frac{\partial}{\partial \theta_1}$ and $\frac{\partial}{\partial \theta_2}$.
    \item The metric on $ (0,r_0 ] \times [0,\frac{\pi}{2} ] \times [0,\frac{\pi}{2} ] \times [0,\frac{\pi}{2} ] $,
    $$ g_{\Phi , \Psi , \Upsilon } = d\gamma^2 + \Phi (\gamma ,\theta) d\theta^2 + \Psi (\gamma ,\theta) d\theta_1^2 + \Upsilon (\gamma ,\theta) d\theta_2^2  $$
    gives a $\mathbb{T}^2$-invariant smooth Riemannian metric on $(0,r_0 ] \times \mathbb{S}^3$, where the $\mathbb{T}^2$-action is induced by $\frac{\partial}{\partial \theta_1}$ and $\frac{\partial}{\partial \theta_2}$.
\end{itemize}
\end{fact}

Let $Y_1=\frac{\partial}{\partial \gamma}$, $Y_2= \frac{1}{\Phi} \frac{\partial}{\partial \theta} $, $Y_3= \frac{1}{\Psi} \frac{\partial}{\partial \theta_1} $ and $Y_4= \frac{1}{\Upsilon} \frac{\partial}{\partial \theta_2} $. Then $\{ Y_1 ,Y_2 ,Y_3 ,Y_4 \}$ is an orthonormal basis. For abbreviation, we will subscript the variable that we differentiated with respect to. For example, we abbreviate $\frac{\partial^2 \Psi }{\partial \gamma \partial \theta }$ to $\Psi_{\gamma \theta}$ or $\Psi_{ \theta \gamma}$. By direct calculation, we obtain:
\begin{eqnarray}
\Ric_{g_{\Phi , \Psi , \Upsilon }} (Y_1) & = & -\left( \frac{\Phi_{\gamma\gamma}}{\Phi} + \frac{\Psi_{\gamma\gamma}}{\Psi}  + \frac{\Upsilon_{\gamma\gamma}}{\Upsilon} \right) Y_1 - \left( \frac{1}{\Psi } \frac{\partial}{\partial \gamma} \left( \frac{\Psi_\theta}{\Phi} \right) + \frac{1}{\Upsilon } \frac{\partial}{\partial \gamma} \left( \frac{\Upsilon_\theta }{\Phi} \right) \right) Y_2 , \\
\Ric_{g_{\Phi , \Psi , \Upsilon }} (Y_2) & = & - \left( \frac{1}{\Psi } \frac{\partial}{\partial \gamma} \left( \frac{\Psi_\theta}{\Phi} \right) + \frac{1}{\Upsilon } \frac{\partial}{\partial \gamma} \left( \frac{\Upsilon_\theta }{\Phi} \right) \right) Y_1 ,\\
& & + \left( -\frac{\Phi_{\gamma\gamma}}{\Phi} - \frac{\Psi_{\theta\theta}}{\Phi^2 \Psi} - \frac{\Upsilon_{\theta\theta}}{\Phi^2 \Upsilon} + \frac{\Phi_\theta \Psi_\theta }{\Phi^3 \Psi} + \frac{\Phi_\theta \Upsilon_\theta}{\Phi^3 \Upsilon} - \frac{\Phi_\gamma \Psi_\gamma}{\Phi \Psi} - \frac{\Phi_\gamma \Upsilon_\gamma}{\Phi \Upsilon}  \right) Y_2 ,\nonumber\\
\Ric_{g_{\Phi , \Psi , \Upsilon }} (Y_3) & = & \left( -\frac{\Psi_{\gamma\gamma}}{\Psi} - \frac{\Psi_{\theta\theta}}{\Phi^2 \Psi} +\frac{\Phi_\theta \Psi_\theta }{\Phi^3 \Psi} - \frac{\Phi_\gamma \Psi_\gamma}{\Phi \Psi} - \frac{\Psi_\gamma \Upsilon_\gamma }{\Psi\Upsilon} - \frac{\Psi_\theta \Upsilon_\theta}{\Phi^2 \Psi\Upsilon } \right) Y_3 ,\\
\Ric_{g_{\Phi , \Psi , \Upsilon }} (Y_4) & = & \left( -\frac{\Upsilon_{\gamma\gamma}}{\Upsilon} - \frac{\Upsilon_{\theta\theta}}{\Phi^2 \Upsilon} +\frac{\Phi_\theta \Upsilon_\theta }{\Phi^3 \Upsilon} - \frac{\Phi_\gamma \Upsilon_\gamma}{\Phi \Upsilon} - \frac{\Psi_\gamma \Upsilon_\gamma }{\Psi\Upsilon} - \frac{\Psi_\theta \Upsilon_\theta}{\Phi^2 \Psi\Upsilon } \right) Y_4 .
\end{eqnarray}

Now we are ready to add a conical singular point on $\mathcal{W}_\mu = W_{\frac{\sin (2r)}{2}} \times W_{\frac{\rho_\mu (r)}{n}}$.

\begin{lmm}\label{addaconicalsingularpoint}
There are constants $\mu_0 = \mu_0 (r_0) \in (0,\frac{1}{10}) $, $\sigma = \sigma (r_0) \in (0,r_0) $ and $\zeta = \zeta (r_0) \in (0,1) $ satisfying the following properties. For any $\mu \in (0,\mu_0 ) $, there are a $\mathbb{T}^2$-invariant metric space $(\mathcal{U} ,d_{\mathcal{U}})$, a compact subset $K_\mu \subset \mathcal{U} $, a point $p\in K$, and a Riemannian manifold $(\mathcal{U} \sq\{p\} ,g_{\mathcal{U}} )$ with $\Ric_{g_{\mathcal{U}}} \geq 0 $, such that
\begin{enumerate}[\rm (i).]
    \item The restriction of $d_{\mathcal{U}}$ on $\mathcal{U} \sq\{p\}$ is given by $g_{\mathcal{U}}$,
    \item The Riemannian manifold $(\mathcal{U} \sq K ,g_{\mathcal{U}} , \mathbb{T}^2 )$ is equivariant isometric to $\left( \mathcal{W}_\mu \sq Z_\mu , \mathbb{T}^2 \right) $ for some $\mathbb{T}^2$-invariant compact subset $Z_\mu \subset \mathcal{W}_\mu $,
    \item The metric ball $(B_{\frac{\sigma}{2}} (p) ,d_{\mathcal{U}} , \mathbb{T}^2 )$ is equivariant isometric to the metric space given by
    $$ g_{\mathcal{U} ,\sigma ,\mu } = d\gamma^2 + (1-\zeta)^2 \gamma^2 \left( d\theta^2 + \frac{\sin^2 (2\sigma \cos \theta )}{4\sigma^2 } d\theta_1^2 + \frac{\rho_\mu^2 (\sigma \sin \theta ) }{n^2 \sigma^2 } d\theta_2^2 \right) $$
    on $(0,r_0 ] \times \mathbb{S}^3$, where $\rho_\mu $ is the function in Proposition \ref{propositionlocalproductstructure}.
\end{enumerate}
\end{lmm}

\begin{proof}
Since the renormalized metric sphere of $(0,0)\in \mathcal{W}_\mu$ does not satisfy $\Ric >1$, our construction here is divide into two parts. The first part is to shrink the metric sphere, and the second part is to modify the metric inside the metric ball to the cone metric.

We now consider the metric $g_{\Phi , \Psi , \Upsilon }$ constructed above under the restriction $\Phi (\gamma ,\theta) = \varphi (\gamma) $, $ \Psi (\gamma ,\theta) = \varphi (\gamma) \phi (\gamma ,\theta) $ and $\Upsilon (\gamma ,\theta) = \varphi (\gamma) \upsilon (\gamma ,\theta) $. In this case, we have
\begin{eqnarray}
\Ric_{g_{\Phi , \Psi , \Upsilon }} (Y_1) & = & -\left( 3\frac{\varphi''}{\varphi} + 2\frac{\varphi' \psi_\gamma}{\varphi \psi} + 2\frac{ \varphi' \upsilon_\gamma}{\varphi \upsilon} + \frac{\psi_{\gamma\gamma}}{\psi} + \frac{\upsilon_{\gamma\gamma}}{\upsilon} \right) Y_1 - \left( \frac{\psi_{\gamma\theta}}{\varphi\psi} + \frac{\upsilon_{\gamma\theta}}{\varphi\upsilon} \right) Y_2 , \\
\Ric_{g_{\Phi , \Psi , \Upsilon }} (Y_2) & = & - \left( \frac{\psi_{\gamma\theta}}{\varphi\psi} + \frac{\upsilon_{\gamma\theta}}{\varphi\upsilon} \right) Y_1 ,\\
& & + \left( -\frac{\varphi''}{\varphi} - \frac{\psi_{\theta\theta}}{\varphi^2 \psi} - \frac{\upsilon_{\theta\theta}}{\varphi^2 \upsilon} - 2\frac{\varphi'^2 }{\varphi^2 } - \frac{\varphi' \psi_\gamma}{\varphi \psi} - \frac{\varphi' \upsilon_\gamma}{\varphi \upsilon}  \right) Y_2 ,\nonumber\\
\Ric_{g_{\Phi , \Psi , \Upsilon }} (Y_3) & = & \left( -\frac{\varphi''}{\varphi} - 2\frac{\varphi'^2}{\varphi^2} - 4\frac{\varphi'\psi_\gamma}{\varphi \psi} -\frac{\psi_{\gamma\gamma}}{\psi} - \frac{\psi_{\theta\theta}}{\varphi^2 \psi} - \frac{\varphi'\upsilon_\gamma}{\varphi\upsilon} -\frac{\psi_\gamma \upsilon_\gamma}{\psi\upsilon} - \frac{\psi_\theta \upsilon_\theta}{\varphi^2 \psi\upsilon } \right) Y_3 ,\\
\Ric_{g_{\Phi , \Psi , \Upsilon }} (Y_4) & = & \left( -\frac{\varphi''}{\varphi} - 2\frac{\varphi'^2}{\varphi^2} - 4\frac{\varphi'\upsilon_\gamma}{\varphi \upsilon} -\frac{\upsilon_{\gamma\gamma}}{\upsilon} - \frac{\upsilon_{\theta\theta}}{\varphi^2 \upsilon} - \frac{\varphi'\psi_\gamma}{\varphi \psi} - \frac{\psi_\gamma \upsilon_\gamma}{\psi\upsilon} - \frac{\psi_\theta \upsilon_\theta}{\varphi^2 \psi\upsilon } \right) Y_4 .
\end{eqnarray}

Write $\varphi_\mathcal{W} = \gamma$, $\psi_\mathcal{W} = \frac{\sin (2\gamma \cos \theta)}{2\gamma} $ and $\upsilon_\mathcal{W} 
= \frac{\rho_\mu (\gamma \sin \theta)}{n\gamma} $. Then the metric $g_\mathcal{W}$ on $\mathcal{W}_\mu$ is just the metric $g_{\Phi , \Psi , \Upsilon }$ in the case $\varphi = \varphi_\mathcal{W} $, $\psi = \psi_\mathcal{W} $ and $\upsilon = \upsilon_\mathcal{W} $.

\smallskip

\par {\em Part 1:} First, we choose $\psi = \psi_\mathcal{W} $ and $\upsilon = \upsilon_\mathcal{W} $. By a direct calculation, one can obtain
\begin{eqnarray}
\Ric_{g_{\Phi , \Psi , \Upsilon }} (Y_1 ,Y_1) & = & -3\frac{\varphi''}{\varphi} + 4\cos^2 \theta - \frac{\rho''_\mu}{\rho_\mu} \sin^2\theta \\
& & + \left( \frac{2}{\gamma^2} - \frac{2\varphi'}{\gamma \varphi} \right) \left( 2-2\cot (2\gamma \cos \theta ) \gamma \cos \theta - \frac{\rho_\mu' \gamma \sin\theta}{\rho_\mu} \right) ,\nonumber\\
\Ric_{g_{\Phi , \Psi , \Upsilon }} (Y_1 ,Y_2) & = & - \frac{\gamma}{\varphi} \left( 4 + \frac{\rho_\mu''}{\rho} \right) \sin \theta \cos \theta ,\\
\Ric_{g_{\Phi , \Psi , \Upsilon }} (Y_2 ,Y_2) & = & -\frac{\varphi''}{\varphi} + \frac{\gamma^2 }{\varphi^2 } \left( 4\sin^2 \theta - \frac{\rho''_\mu}{\rho_\mu} \cos^2\theta \right) + 2\frac{\varphi'}{\varphi} \left( \frac{1}{\gamma} - \frac{\varphi'}{\varphi} \right) \\
& & + \left( \frac{2}{\varphi^2} - \frac{2\varphi'}{\gamma \varphi} \right) \left( 2-2\cot (2\gamma \cos \theta ) \gamma \cos \theta - \frac{\rho_\mu' \gamma \sin\theta}{\rho_\mu} \right) ,\nonumber\\
\Ric_{g_{\Phi , \Psi , \Upsilon }} (Y_3 ,Y_3) & = & -\frac{\varphi''}{\varphi} + 4 \left( \frac{\gamma^2}{\varphi^2} \sin^2 \theta + \cos^2\theta \right) - \left( \frac{\varphi'}{\varphi} -\frac{1}{\gamma} \right) \left( \frac{2\varphi'}{\varphi} -\frac{3}{\gamma} \right) \\
& & + 2 \left( \frac{\gamma}{\varphi^2} -\frac{4\varphi'}{\varphi} +\frac{3}{\gamma} \right) \cot (2\gamma \cos \theta ) \cos \theta + \left( \frac{1}{\gamma} -\frac{ \varphi'}{\varphi} \right) \frac{\rho'_\mu}{\rho_\mu} \sin\theta \nonumber \\
& & + \left( \frac{\gamma^2}{\varphi^2} -1 \right) \frac{2\rho'_\mu}{\rho_\mu} \cot (2\gamma \cos \theta ) \cos \theta \sin\theta ,\nonumber \\
\Ric_{g_{\Phi , \Psi , \Upsilon }} (Y_4 ,Y_4) & = &  -\frac{\varphi''}{\varphi} - \frac{\rho_\mu''}{\rho_\mu} \left( \frac{\gamma^2}{\varphi^2} \cos^2 \theta + \sin^2\theta \right) - \left( \frac{\varphi'}{\varphi} -\frac{1}{\gamma} \right) \left( \frac{2\varphi'}{\varphi} -\frac{3}{\gamma} \right) \\
& & + 2 \left( \frac{1}{\gamma} -\frac{ \varphi'}{\varphi} \right) \cot (2\gamma \cos \theta ) \cos \theta + \left( \frac{\gamma}{\varphi^2} -\frac{4\varphi'}{\varphi} +\frac{3}{\gamma} \right) \frac{\rho'_\mu}{\rho_\mu} \sin\theta \nonumber \\
& & + \left( \frac{\gamma^2}{\varphi^2} -1 \right) \frac{2\rho'_\mu}{\rho_\mu} \cot (2\gamma \cos \theta ) \cos \theta \sin\theta \nonumber .
\end{eqnarray}
In the above equations, the independent variables of the functions with respect to $\varphi$ are all taken as $\gamma$, and the independent variables of the functions with respect to $\rho_\mu$ are all taken as $\rho_\mu (\gamma \sin\theta )$, so we can omit both for convenience. For example, we abbreviate $\rho_\mu' (\gamma \sin\theta )$ to $\rho_\mu'$.

By the construction of $\rho_\mu$, one can see that $\left|\frac{\rho'_\mu}{\rho_\mu} \gamma \sin\theta \right| + | \gamma \cos \theta \cot (2\gamma \cos \theta ) | \leq 2 + \mu $. Fix a small constant $\zeta = \zeta (r_0) >0 $ and a function $\varphi_1 \in C^\infty ([0,r_0 ]) $ such that
\begin{itemize}
    \item $ |\varphi''_1| + |\varphi_1 - \gamma  \varphi'_1 | \leq \frac{r_0^4}{100} $,
    \item $(1-\zeta ) \gamma \leq \varphi_1 (\gamma) \leq \gamma $,
    \item $\varphi_1 (\gamma) = (1-\zeta ) \gamma $ on $[0,\frac{r_0}{2}]$,
    \item $\varphi_1 (\gamma) = \gamma $ on $[ \frac{r_0}{2} ,r_0 ]$.
\end{itemize}
Now we return to estimating the lower bound of Ricci curvature. Let $\varphi=\varphi_1$. Then the above construction implies that $\Ric_{g_{\Phi , \Psi , \Upsilon }} (Y_3 ,Y_3) \geq 3 $ and $\Ric_{g_{\Phi , \Psi , \Upsilon }} (Y_4 ,Y_4) \geq 3 $. For $Y_1$ and $Y_2$, note that the Ricci curvature here has a mixing term, so we need to control the mixing term. By the construction above again, we have
\begin{eqnarray}
\Ric_{g_{\Phi , \Psi , \Upsilon }} (Y_1 ,Y_1) & = &  4\cos^2 \theta - \frac{\rho''_\mu}{\rho_\mu} \sin^2\theta - \frac{1}{10} >0 ,\\
\Ric_{g_{\Phi , \Psi , \Upsilon }} (Y_2 ,Y_2) & = & \frac{\gamma^2 }{\varphi^2 } \left( 4\sin^2 \theta - \frac{\rho''_\mu}{\rho_\mu} \cos^2\theta  -\frac{1}{10} \right) >0 ,
\end{eqnarray}
and hence
\begin{eqnarray}
& &  \frac{\varphi^2}{\gamma^2} \left( \Ric_{g_{\Phi , \Psi , \Upsilon }} (Y_1 ,Y_1) \Ric_{g_{\Phi , \Psi , \Upsilon }} (Y_2 ,Y_2) - \Ric_{g_{\Phi , \Psi , \Upsilon }} (Y_1 ,Y_2)^2 \right) \nonumber \\
& \geq & \left(4\cos^2 \theta - \frac{\rho''_\mu}{\rho_\mu} \sin^2\theta - \frac{1}{10} \right) \left( 4\sin^2 \theta - \frac{\rho''_\mu}{\rho_\mu} \cos^2\theta  -\frac{1}{10} \right) - \left( 4 + \frac{\rho_\mu''}{\rho} \right)^2 \sin^2 \theta \cos^2 \theta \\
& = & -\frac{39}{10} \frac{\rho_\mu''}{\rho_\mu} - \frac{2}{5} \geq 0. \nonumber
\end{eqnarray}

It follows that $\Ric_{g_{\Phi , \Psi , \Upsilon }} \geq 0 $ if we choose $\varphi=\varphi_1$, $\psi = \psi_\mathcal{W} $ and $\upsilon = \upsilon_\mathcal{W} $.

\smallskip

\par {\em Part 2:} Set $\hat{\sigma} \in (0,\frac{r_0}{10} ) $ be a constant to be specified later. Fix a constant $\delta \in (0,\frac{\hat{\sigma}}{10}) $. Let $\eta_\delta $ be a $C^\infty$ convex function on $ [0 , \frac{r_0}{2} ] $ such that 
\begin{itemize}
    \item $\eta_\delta (\gamma) = \gamma $ on $[\hat{\sigma} , \frac{r_0}{2} ]$,
    \item $\eta_\delta (\gamma) = \hat{\sigma} - \frac{\delta}{2} $ on $[0,\hat{\sigma} - \delta ]$.
\end{itemize}

Let $ \varphi (\gamma) =(1-\zeta ) \gamma $, $ \psi (\gamma ,\theta) = \psi_{\mathcal{W}} (\eta_\delta (\gamma ) ,\theta ) $ and $ \upsilon (\gamma ,\theta) = \upsilon_{\mathcal{W}} (\eta_\delta (\gamma ) ,\theta ) $.

By $ \left|\frac{\rho'_{ \mu } (r) \sin (2 r) }{ 2 \rho_{ \mu } (r) \cos (2 r) } -1 \right| \leq \mu $, one can find constants $\mu_1 , \gamma_1 \in (0,\frac{r_0}{10} ) $ such that if $\mu\in (0,\mu_1 ) $, $\gamma \in (0,\gamma_1) $, then $ \frac{\rho'_{ \mu } (\gamma) }{ \rho_{ \mu } (\gamma) } \gamma \geq 1-\frac{\zeta}{10} $ and $ \gamma \cot \gamma \geq 1-\frac{\zeta}{10} $. Moreover, for each given $\gamma_2 \in (0,\frac{r_0}{10} ) $, there exists a constant $\mu_2 = \mu_2 (\gamma_2) $ such that for any $\gamma \in (\gamma_2 , \frac{r_0}{10} ) $ and $\mu\in (0,\mu_2 )$, we have
\begin{eqnarray}
    \frac{\rho_\mu' (\gamma \sin\theta ) }{\rho_\mu (\gamma \sin\theta )} \gamma\sin\theta \leq 1 .
\end{eqnarray}

Assume that $\hat{\sigma} \in (0,\gamma_1 ) $ and $\mu\in (0,\mu_1) \cap (0,\mu_2 (\hat{\sigma} -\delta ) ) $. Then we have
\begin{eqnarray}
\Ric_{g_{\Phi , \Psi , \Upsilon }} (Y_1 ,Y_1) & = & \left( \frac{2\eta'_{\delta}}{\gamma \eta_{\delta}} - \frac{2\eta'^2_{\delta}}{ \eta^2_{\delta}} + \frac{\eta''_{\delta}}{ \eta_{\delta}} \right) \left( 2- \frac{\rho_\mu' }{\rho_\mu } \eta_{\delta}\sin\theta - 2\cot (2\gamma \cos \theta ) \eta_{\delta} \cos \theta \right) \nonumber\\
& & + \eta_\delta'^2 \left( 4\cos^2 \theta - \frac{\rho_\mu''  }{\rho_\mu } \sin^2 \theta \right) \\
& \geq & \eta_\delta'^2 \left( 4\cos^2 \theta - \frac{\rho_\mu''  }{\rho_\mu } \sin^2 \theta \right) .\nonumber
\end{eqnarray}
In the above inequalities,  the independent variables of the functions with respect to $\eta_\delta $ are all taken as $\gamma$, and the independent variables of the functions with respect to $\rho_\mu$ are all taken as $ \eta_\delta (\gamma) \sin\theta $, so we can omit both for convenience. For example, we abbreviate $\rho_\mu' (\eta_\delta (\gamma) \sin\theta )$ to $\rho_\mu'$.

Similarly, one can obtain
\begin{eqnarray}
\Ric_{g_{\Phi , \Psi , \Upsilon }} (Y_1 ,Y_2) & = & - \frac{\eta_\delta \eta'_\delta}{(1-\zeta ) \gamma} \left( 4 + \frac{\rho_\mu''}{\rho} \right) \sin \theta \cos \theta ,\\
\Ric_{g_{\Phi , \Psi , \Upsilon }} (Y_2 ,Y_2) & \geq & \frac{ \eta_\delta^2}{(1-\zeta )^2 \gamma^2} \left( 4\sin^2 \theta - \frac{\rho''_\mu}{\rho_\mu} \cos^2\theta  \right) ,
\end{eqnarray}
and hence
\begin{eqnarray}
\Ric_{g_{\Phi , \Psi , \Upsilon }} (Y_1 ,Y_1) \Ric_{g_{\Phi , \Psi , \Upsilon }} (Y_2 ,Y_2) - \Ric_{g_{\Phi , \Psi , \Upsilon }} (Y_1 ,Y_2)^2 & 
\geq & 0 .
\end{eqnarray}

Now we estimate $\Ric_{g_{\Phi , \Psi , \Upsilon }} (Y_3 ,Y_3)$ and $\Ric_{g_{\Phi , \Psi , \Upsilon }} (Y_4 ,Y_4)$. Assume that $\hat{\sigma} \in (0,\gamma_1 ) $ and $\mu\in (0,\mu_1) \cap (0,\mu_2 (\hat{\sigma} -\delta ) ) $. Hence we have
\begin{eqnarray}
\Ric_{g_{\Phi , \Psi , \Upsilon }} (Y_3 ,Y_3) & = & \frac{ 1 }{(1-\zeta )^2 \gamma^2} \left( 2 \frac{\rho_\mu' }{\rho_\mu } \eta_{\delta} \cot (2\gamma \cos \theta ) \eta^2_{\delta} \sin \theta \cos \theta - (1-\zeta )^2 \right) \nonumber\\
& & + \frac{ 1 }{(1-\zeta )^2 \gamma^2} \left( 2 \cot (2\gamma \cos \theta ) \eta_{\delta} \cos \theta - (1-\zeta )^2 \right) \nonumber\\
& & + 4 \frac{\eta'_\delta}{\gamma \eta_\delta} \left( 1 - 2 \cot (2\gamma \cos \theta ) \eta_{\delta} \cos \theta \right) + \frac{\eta'_\delta}{\gamma \eta_\delta} \left( 1 - \frac{\rho_\mu' }{\rho_\mu } \eta_{\delta} \sin \theta \right) \nonumber\\
& & + \frac{\eta'^2_\delta}{ \eta^2_\delta} \left( 6 \cot (2\gamma \cos \theta ) \eta_{\delta} \cos \theta -3 +4 \cos^2 \theta \right) + 4 \frac{\eta_\delta^2}{(1-\zeta)^2 \gamma^2} \sin^2 \theta \\
& & + \left( \frac{\eta'^2_\delta}{ \eta^2_\delta} \frac{\rho_\mu' }{\rho_\mu } \eta_{\delta} \sin \theta + \frac{\eta''_\delta}{ \eta_\delta} \right)  \left( 1- 2 \cot (2\gamma \cos \theta ) \eta_{\delta} \cos \theta \right) \nonumber\\
& \geq & \frac{\zeta}{2(1-\zeta)^2 \gamma^2} - \frac{\mu}{\gamma^2} - \frac{3\mu}{\gamma^2} = \frac{4}{ \gamma^2} \left( \frac{\zeta}{8(1-\zeta)^2 } -  \mu \right) .\nonumber
\end{eqnarray}
Hence $\Ric_{g_{\Phi , \Psi , \Upsilon }} (Y_3 ,Y_3) \geq 0 $ when $ \mu \in (0, \frac{\zeta}{8(1-\zeta)^2 } ) $. Similarly, there exists a constant $\mu_3 = \mu_3 (\zeta) \in (0, \frac{\zeta}{8(1-\zeta)^2 } ) $ such that for any $\hat{\sigma} \in (0,\gamma_1 ) $ and $\mu\in (0,\mu_1) \cap (0,\mu_2 (\hat{\sigma} -\delta ) ) \cap (0, \mu_3 ) $, we have $\Ric_{g_{\Phi , \Psi , \Upsilon }} (Y_4 ,Y_4) \geq 0 $. Then we can prove this lemma by choosing $\sigma = \hat{\sigma} - \delta $ and $\mu_0 = \frac{1}{2} \min \{ \mu_1 ,\mu_2 (\hat{\sigma} -\delta ) , \mu_3 \} $.
\end{proof}

\begin{rmk}
In our construction, one can see that the metric
$$ g_{\mathbb{S}^3 ,\sigma ,\mu } = (1-\zeta)^2 \left( d\theta^2 + \frac{\sin^2 (2\sigma \cos \theta )}{4\sigma^2 } d\theta_1^2 + \frac{\rho_\mu^2 (\sigma \sin \theta ) }{n^2 \sigma^2 } d\theta_2^2 \right) $$
on $ \mathbb{S}^3$ satisfies $\Ric_{g_{\mathbb{S}^3 ,\sigma ,\mu }} \geq \left( 2+\frac{\zeta}{100} \right) g_{\mathbb{S}^3 ,\sigma ,\mu } $.
\end{rmk}

\subsection{Modify the tangent cones}
In the previous step, we have largely succeeded in constructing the space we need to use for desingularization. All we have to do now is to modify the tangent cone at $(0,0,[0,1])$ and the asymptotic cone at infinity to cones over round spheres. Note that we can use the following construction given by Colding-Naber \cite{coldingnaber2} in this step.

\begin{lmm}[{\cite[Lemma 2.1]{coldingnaber2}}]
\label{lmmcoldingnaberexamplelemma}
Let $X^{n-1}$ be a smooth compact manifold with $g(s)$, a smooth family of metrics $s\in [0,1] $ such that
\begin{enumerate}[\rm i).]
\item $\Ric_{g(s)} \geq (n-2) g(s) $.
\item $\frac{d}{ds} dv(g(s)) =0 $, where $dv$ is the associated volume form.
\end{enumerate}
Then for any pair $(\lambda_1,\lambda_2)$ such that $1\geq \lambda_1>\lambda_2>0$, there exist a constant $R>0$ and a metric $d$ on $(C(X) ,o) = ([0,\infty )\times X /\sim ,0) $ given by $dr^2 + f(r)^2 g(h(r)) $, such that $B_1 (o) \subset (C(X) ,d,o) $ is isometric to the unit ball $B_1 (o) \subset (C(X) ,dr^2 + \lambda_1 r^2 g(0) ,o) $, and the annulus $	A_{R,\infty } (o) \subset (C(X) ,d,o) $ is isometric to the annulus $	A_{R,\infty } (o) \subset (C(X) ,dr^2 + \lambda_2 r^2 g(1) ,o) $. 
\end{lmm}

The proof of Lemma \ref{lmmcoldingnaberexamplelemma} is exactly the same as the proof of \cite[Lemma 2.1]{coldingnaber2}. 

Now we begin to modify the tangent cone at $(0,0,[0,1])$. Due to Colding-Naber's lemma, we only need to connect $g_{\mathbb{S}^3 ,\sigma ,\mu }$ and the round sphere by a family of metrics.

For any $s\in [0,1]$, set 
\begin{eqnarray}
\hat{g} (s) & = & \left( 1-\frac{999\zeta}{1000} \right)^2 \left( d\theta^2 + \left( (1-s) \cos \theta + s \frac{\sin  (2\sigma \cos \theta )}{2\sigma } \right)^2 d\theta_1^2 \right.\\
& & \left.\;\;\;\;\;\;\;\;\;\;\;\;\;\;\;\;\;\;\;\;\;\;\;\;\;\;\;\;\;\; + \left( (1-s) \sin \theta + s \frac{\rho_\mu (\sigma \sin \theta ) }{n \sigma } \right)^2 d\theta_2^2 \right) .\nonumber
\end{eqnarray}
It is easy to see that $\Ric_{\hat{g} (s)} \geq 2 \hat{g} (s) $, and $\Vol (\hat{g}(s))$ is decreasing. Then the family of metrics $\tilde{g}(s) = \frac{\Vol (\hat{g}(1)) }{\Vol (\hat{g}(s)) } \hat{g}(s) $ have the same volume. Hence there exists a family of differeomorphisms $ \mathfrak{F} (s,x) : [0,1] \times \mathbb{S}^3 \to \mathbb{S}^3 $ such that $\mathfrak{F} (1,x)=x$, and the volume form of $\mathfrak{F} (s,\cdot)^* \tilde{g}(s) $ is independent of $s$. Since $\hat{g} (s)$ are $\mathbb{T}^2$-invariant, one can see that $\mathfrak{F} (s,\cdot)^* \tilde{g}(s) $ is also $\mathbb{T}^2$-invariant. Hence we can choose $\lambda_1 =1 $, $\lambda_2 = \frac{1000-1000\zeta}{1000-999\zeta} $, and $g(s)=\mathfrak{F} (s,\cdot)^* \tilde{g}(s) $.

Then we consider the asymptotic cone.

By a classical result of Hamilton \cite{rhamilton1}, for any $3$-manifold $(M,g)$ with $\Ric_g >0$, there exists a family of metrics $g(t) $ on $M$ such that $g(0)=g$, $\Ric_{g(t)} >0 $, and the rescaled metrics $\tilde{g(t)}$ converge to a space form $\mathbb{S}^3 /G $ as $t\to T$ in the Cheeger-Gromov sense. Then one can use Lemma \ref{lmmcoldingnaberexamplelemma}.

Combining Proposition \ref{propositionlocalproductstructure} with the argument above, one can prove the following result, and Proposition \ref{propinductioncyclic} follows.

\begin{prop}\label{propinductioncyclicdetailedver}
Let $(n,p)$ be a coprime pair of integers, and $n>p\geq 1$. Then for any $\delta >0$, there exist
\begin{itemize}
    \item A Riemannian metric $g_{n,p,\delta }$ on $(\mathcal{O}(-n) /\Gamma_{p,0,1} ) \sq ( \proj^{-1} ([1,0]) \cup \{(0,0,[0,1]) \} ) $,
    \item Positive constants $\epsilon = \epsilon (n,p,\delta) $, $r_\delta = r_\delta (n,p,\delta) $, $R_\delta = R_\delta (n,p,\delta) $,
    \item A compact subset $K_\delta \subset (\mathcal{O}(-n) /\Gamma_{p,0,1} ) $ containing the singular point $ (0,0,[0,1]) $,
\end{itemize}
such that
\begin{itemize}
    \item The extension of the metric $g_{n,p,\delta }$ on $(\mathcal{O}(-n) /\Gamma_{p,0,1} ) \sq \{ (0,0,[0,1]) \}$ gives a smooth Riemannian manifold structure on $(\mathcal{O}(-n) /\Gamma_{p,0,1} ) \sq \{ (0,0,[0,1]) \}$ with $\Ric_{g_{n,p,\delta}} \geq 0$,
    \item The Riemannian manifold $((\mathcal{O}(-n) /\Gamma_{p,0,1} ) \sq K_\delta ,g_{n,p,\delta} )$ is isometric to the annulus $A_{R_\delta ,\infty } (o) $ in the metric cone $ (C(\mathbb{S}_{\epsilon}^3 /\Gamma_{n,1,p} ) ,o ) $,
    \item The metric ball $B_{r_\epsilon} ((0,0,[0,1])) \subset \mathcal{O}(-n) /\Gamma_{p,0,1} $ is isometric to the metric ball $B_{r_\epsilon} (o) \subset C(\mathbb{S}_{\delta}^3 / \Gamma_{p,-1,n } ) $,
\end{itemize}
where $\proj^{-1} ([1,0]) \subset \mathcal{O}(-n) $ is the fiber over $[1,0] \in \mathbb{C}P^1 $.
\end{prop}

\section{Resolution surgery with \texorpdfstring{$\Ric \geq 0$}{Lg} : the general case}
\label{sectionsurgeryriccigeneral}

In this section, we will give another resolution surgery with $\Ric \geq 0$ can be used in the non-cyclic case. The statement of this surgery in the non-cyclic case is similar to Proposition \ref{propinductioncyclic}.

\begin{prop}\label{propinductiongeneral}
Let $\Gamma $ be a non-trivial finite subgroup of $U(2) \subset O(4)$ acts free on the unit sphere $\mathbb{S}^3$. Then for any $\epsilon >0$, we can find finite cyclic subgroups $\hat{\Gamma}_k$ of $U(2) $ acts free on the unit sphere $\mathbb{S}^3$, a metric space $(\mathcal{X},d )$, and points $x_k \in X$, $k=1,\cdots ,N$, satisfying the following properties:
\begin{itemize}
\item $|\hat{\Gamma}_k| \leq |\Gamma |$, $k=1,\cdots ,N$,
\item The metric subspace $(\mathcal{X}\sq \{x_k\}_{1\leq k\leq N} ,d)$ is isometric to a Riemannian manifold $(M,g)$ with nonnegative Ricci curvature,
\item The asymptotic cone of $(\mathcal{X},d)$ is isometric to $C(\mathbb{S}^3_{\delta } /\Gamma )$ for some $\delta =\delta (\Gamma ,\epsilon ) >0$,
\item The metric ball $(B_1 (x_k) ,d)$ is isometric to $B_1 (o) \subset ( C(\mathbb{S}^3_{\epsilon } /\hat{\Gamma}_k ) ,o)$.
\end{itemize}
\end{prop}

\begin{rmk}
In particular, if $\Gamma \cap Z(U(2)) \neq 0$, then we have $|\hat{\Gamma}_k| < |\Gamma |$, $k=1,\cdots ,N$, where $Z(U(2))$ is the center of $U(2)$.
\end{rmk}

Similar to Proposition \ref{propinductioncyclic}, the topological space in Proposition \ref{propinductiongeneral} is also the topological space given in Section \ref{sectionblowup}. Since many of the arguments in the proof of Proposition \ref{propinductiongeneral} are similar to the corresponding parts of the proof of Proposition \ref{propinductioncyclic}, we give here only an outline of the proof.

\begin{proof}
At first , we describe how to construct the ambient space.

By the construction in Section \ref{sectionblowup}, the space we consider here is the orbifold $\mathcal{O}(-n)/\tilde{\Gamma}$ with discrete singularities, where $\tilde{\Gamma} \cong \Gamma /\Gamma \cap Z(U(2)) $, and $Z(U(2))$ is the center of $ U(2) $. Write $S=\{ x\in \mathcal{O}(-n) : \tilde{\Gamma}_x \neq \{\rm id\}  \}$, where $\tilde{\Gamma}_x$ is the stabilizer of $x$. Clearly, $S$ is a discrete subset of the base $\mathbb{C}P^1 \subset \mathcal{O}(-n) $.

Now we can construct a Riemannian metric on $\mathcal{O}(-n)/\tilde{\Gamma}$ by considering the standard Berger sphere. Let $\pi_{Hopf} :\mathbb{S}^3 \to \mathbb{S}^2_{\frac{1}{2}} $ be the Hopf fibration as in Section \ref{sectionsurgeryriccicyclic}, $U$ be the unit tangent vector field along the Hopf fibers, and $\{ U,X,Y \}$ be an orthonormal basis of the tangent bundle of $\mathbb{S}^3$. Then for any smooth functions $\rho (r) $, $\phi (r)$ on $[0,\infty )$ satisfying that
\begin{itemize}
    \item $\rho >0$ on $(0,\infty )$, $\rho (0)=\rho^{even} (0) =0$, $\rho'(0) =n $,
    \item $\phi >0$ on $[0,\infty )$, $\phi^{odd} (0)= 0$,
\end{itemize}
the metric
$$ g_{\rho ,\phi} = dr^2 + \rho (r)^2 U^{*2} + \phi(r)^2 ( X^{*2} + Y^{*2} ) $$
gives a Riemannian metric on $\mathcal{O} (-n)$, and it is easy to see that $g_{\rho ,\phi}$ is invariant under the action of $\tilde{\Gamma}$, where $\{ U^* ,X^* ,Y^* \}$ is the dual basis of $\{ U,X,Y \}$.

As in Section \ref{sectionsurgeryriccicyclic}, by a direct calculation, one see that
\begin{eqnarray}
 \Ric_{g_{\rho ,\phi}} ( \frac{\partial}{\partial r} , \frac{\partial}{\partial r} ) & = & -\frac{\rho''}{\rho} -2\frac{\phi''}{\phi} ,\\
 \Ric_{g_{\rho ,\phi}} ( U , U ) & = & 2\frac{\rho^4}{\phi^4} - 2\frac{\rho\rho'\phi'}{\phi} - \rho \rho'' ,\\
 \Ric_{g_{\rho ,\phi}} ( X , X ) & = & \Ric_{g_{\rho ,\phi}} ( Y , Y ) = 4 -2 \frac{\rho^2}{\phi^2} - \frac{\rho\rho'\phi'}{\phi} - \phi\phi'' - (\phi')^2 ,
\end{eqnarray}
and the mixed terms of $\Ric_{g_{\rho ,\phi}}$ are equal to $0$.

Then we can use the way in the proof of Lemma \ref{lemmaedgesingularities} to construct functions $\rho_{\mu}$ and $\phi_{\mu}$ such that
\begin{itemize}
    \item $\Ric_{g_{\rho ,\phi}} \geq 0$,
    \item For any $r\in \left( 0,\frac{1}{10 } \right) $, $\phi_{ \mu} (r) =1$,
    \item For any $r\in \left( 0,\frac{1}{10 } \right) $, $\frac{\rho_{ \mu} (r)}{n} \leq \min \{ \sin ( r) , \mu + \mu \sin ( r) \} $, $ \frac{\rho'_{ \mu } (r) \sin (2 r) }{ 2 \rho_{ \mu } (r) \cos (2 r) } \in [1-\mu ,1+\mu] $, and $ \frac{-\rho''_{ \mu } (r) }{ 4 \rho_{ \mu } (r) } \geq 1-\mu $,
    \item There are constants $c_1, c_2 ,c_3 >0$ and $R>0$ such that for any $r\geq R$, $\rho_{ \mu} (r) = c_1 (r+ c_3 ) $, and $\phi_{ \mu} (r) = c_2 (r+ c_3 ) $.
\end{itemize} 

Our next step is to adding conical singularities on $x\in S \subset \mathcal{O} (-n) $. The argument in this step follows almost verbatim from the proof of Proposition \ref{propositionlocalproductstructure} and Lemma \ref{addaconicalsingularpoint}. Then we can find a constant $\mu_0$ such that for any $\mu\in (0,\mu_0 )$, one can modify the metric at each point in $ S $.

Finally we just need to modify the cones we constructed to cones over round spheres, and this part is the same as the proof of Proposition \ref{propinductioncyclic}.
\end{proof}

\section{Proof of Theorem \ref{thmconstructionO4}}
\label{sectionproofmainthm}
In this section, we will prove Proposition \ref{propconstructionU2} and Theorem \ref{thmconstructionO4}.

For convenience, let us restate it here.

\begin{prop}[{= Proposition \ref{propconstructionU2}}]
Let $\Gamma $ be a finite subgroup of $U(2) \subset O(4)$ acting freely on the unit sphere $\mathbb{S}^3$, and $M$ be the minimal resolution of $\mathbb{C}^2 /\Gamma $. Then there exist a complete Riemannian metric $g$ on $M$ and a compact subset $K\subset M$ such that $\Ric_g \geq 0 $, and the manifold $(M\sq K,g)$ is asymptotic to the annulus $A_{1,\infty} (o) \subset (C(\mathbb{S}_\delta^3 /\Gamma ) ,o)$ for some $\delta = \delta (\Gamma ) >0$.
\end{prop}

\begin{proof}
We prove this proposition for cyclic $\Gamma $ by induction on $|\Gamma|$.

If $|\Gamma| =1 $, $M$ is diffeomorphic to $\mathbb{R}^4$, and this proposition is trivial.

Now we assume that this proposition holds when $|\Gamma| \leq n-1 $. 

Suppose that $|\Gamma | =n $. Then we can apply Proposition \ref{propinductioncyclicdetailedver} to $\Gamma$. In this case, for any $\epsilon >0$, we can find a finite cyclic subgroup $\hat{\Gamma}$ of $U(2) $ acts free on the unit sphere $\mathbb{S}^3$ and a pointed metric space $(X,d,x)$ satisfying the following properties:
\begin{itemize}
\item $|\hat{\Gamma}| \leq n-1$,
\item The metric subspace $(X\sq \{x\} ,d)$ is isometric to a Riemannian manifold $(M,g)$ with nonnegative Ricci curvature,
\item $(X,d)$ is isometric to $C(\mathbb{S}^3_{\delta } /\Gamma )$ for some $\delta =\delta (\Gamma ,\epsilon ) >0$ outside a compact subset,
\item The metric ball $(B_1 (x) ,d)$ is isometric to $B_1 (o) \subset ( C(\mathbb{S}^3_{\epsilon } /\hat{\Gamma} ) ,o)$.
\end{itemize}
By induction hypothesis, we can glue a manifold $M_{\hat{\Gamma}}$ such that $M_{\hat{\Gamma}}$ is isometric to $C(\mathbb{S}^3_{\hat{\delta} } /\hat{\Gamma} )$ for some $\hat{\delta} >0$ outside a compact subset. Then the proposition holds when $|\Gamma| =n $. Note that in our construction, the manifold $X\sq \{x\} \cup M_{\hat{\Gamma}} $ is diffeomorphic to a smooth complex surface without $(-1)$-curve, and hence $X\sq \{x\} \cup M_{\hat{\Gamma}} $ is diffeomorphic the minimal resolution of $\mathbb{C}^2 /\Gamma $ \cite{bhpv1}. It follows that this proposition holds for cyclic $\Gamma $.

Now we assume that $\Gamma$ is not a cyclic group. Then by Proposition \ref{propinductiongeneral}, one can apply the argument above to show that this proposition can be reduced to the cyclic case. This completes the proof.
\end{proof}

As a corollary, we can prove Theorem \ref{thmconstructionO4}.

\vspace{0.2cm}

\noindent \textbf{Proof of Theorem \ref{thmconstructionO4}: }
Let $\Gamma $ be a finite subgroup of $O(4)$ acting freely on the unit sphere $\mathbb{S}^3$. Then Proposition \ref{propunitaryrepresentation} shows that there exists a unitary representation $\pi :\Gamma \to U(2)$ such that $\mathrm{ker} \pi = \{ E \} $, the image $\pi (\Gamma ) < U(2) $ is also a group acting freely on $\sphere^3$, and $\sphere^3 /\Gamma \cong \sphere^3 /\pi (\Gamma )  $ as Riemannian manifolds, where $E\in O(4)$ is the identity matrix. Then we can replace $\Gamma$ by $\pi (\Gamma)$, and one can see that Proposition \ref{propconstructionU2} implies the theorem.
\qed

\begin{rmk}
   Now we briefly describe how to construct infinitely many Riemannian manifolds with non-negative Ricci curvature that are asymptotic to $C(\mathbb{S}_\delta^3 /\Gamma )$. For this, we need to revisit the construction of the function $\rho_{\kappa ,\mu}$ in Lemma \ref{lemmaedgesingularities}. Note that although we set $\varepsilon =\mu $ for convenience in the following steps, we can actually fix a small $\varepsilon$ and shrink the value of $\mu$. Consequently, we can ensure that the metric on the given open subset converges to $ W_{\frac{\sin (2r)}{2}} \times W_{\varepsilon \sin (r) } $. Then we can glue any number of $\mathbb{C}P^2$ on this open subset. The argument here is similar to those of Perelman \cite{perel1} and Menguy \cite{menguy2}.
\end{rmk}

\appendix

\section{Unitary representations of spherical groups} 
 \label{appspherical}

In this appendix we consider the unitary representations of fundamental groups of spherical $3$-manifolds. Our goal is to obtain the following proposition.

\begin{prop}
\label{propunitaryrepresentation}
Let $G < O(4)$ be a finite group acting freely on $\sphere^3$. Then there exists a unitary representation $\pi :G \to U(2)$ such that $\mathrm{ker} \pi = \{ \rm id \} $, the image $\pi (G ) < U(2) $ is also a group acting freely on $\sphere^3$, and $\sphere^3 /\Gamma \cong \sphere^3 /\pi (\Gamma )  $ as Riemannian manifolds, where ${\rm id} \in O(4)$ is the identity matrix.
\end{prop}

\begin{rmk}
Actually, if the finite subgroups $\Gamma_1 $, $ \Gamma_2 $ of $O(4)$ acting freely on $\sphere^3$, $\Gamma_1 \cong \Gamma_2 $, but $\sphere^3 /\Gamma_1 $ is not isometric to $\sphere^3 /\Gamma_2$, then both $\Gamma_1 $ and $ \Gamma_2 $ are cyclic groups.
\end{rmk}

Here we follow the argument in Thurston's book \cite{thur1}. See also \cite[Part III]{wolf1}.

\begin{proof}
Since the group generated by $G$ and $ \{ \pm {\rm id_4} \} < O(4) $ acts freely on $\mathbb{S}^3$, one can conclude that the group $\hat{G} = G/G\cap \{ \pm {\rm id_4} \} $ acts freely on $\mathbb{R}P^3 \cong \mathbb{S}^3 / \{ \pm {\rm id_4} \} $. Note that $\mathbb{R}P^3 \cong SO(3) $ and the action of $\hat{G}$ on $\mathbb{R}P^3 $ is orientation preserving. Then there exists an embedding
$$ \iota = (\iota_1 ,\iota_2 ) : \hat{G} \longrightarrow SO(3) \times SO(3) ,$$
such that the action of $\hat{G}$ on $\mathbb{R}P^3 $ can be given by $(g,x) \mapsto \iota_1 (g) x \iota_2 (g)^{-1} $.

Write $H_1 = \iota_1 (\hat{G} ) $, $H_2 = \iota_2 (\hat{G} ) $, $G_1 = \iota (\hat{G} ) \cap (SO(3) \times \{ {\rm id_3} \}) $, and $G_2 = \iota (\hat{G} ) \cap (\{ {\rm id_3} \} \times SO(3)) $. It is easy to check that $\hat{G} $ acts freely on $SO(3) $ implies that for any $g\in \hat{G}$, $\iota_1 (g)$ is not conjugate to $\iota_2 (g)$ in $O(4)$. Since all elements of order $2$ are conjugate to ${\rm diag } (1,-1,-1)$ in $SO(3)$, we see that $G_1$ and $G_2$ cannot both have elements of order $2$. Without loss of generality, we assume that $G_2$ has no elements of order $2$. Let $g \in \hat{G}$ such that the order of $\iota_1 (g)$ is $2$. Hence $\iota_2 (g^2) \in G_2 $, and the order of $\iota_2 (g^2)  $ is odd. Write $ {\rm ord}(\iota_2 (g^2))=m $. Then ${\rm ord}(\iota_1 (g^m))=2$. Since $\iota_2 (g^m) $ $\iota_1 (g^m) $, one can conclude that the order of $\iota_2 (g^m) $ is not equal to $2$, $\iota_2 (g^m) = {\rm id_3} $, and $\iota_1 (g)= \iota_1 (g^m) \in G_1 $. Hence all elements of order $2$ in $H_1$ are elements in $G_1$. 

According to the famous A-D-E classification of finite subgroups of $SO(3)$, which can be found in \cite[Chapter 1]{klein1} or \cite[Theorem 11]{egrs1}, one can observe that $H_1$ is isometric to one of the following:
\begin{itemize}
\item $\mathbb{Z}_n$, the cyclic group of order $n\geq 2$.
\item $\mathbb{D}_n$, the dihedral group of symmetries of a regular $n$-gon, where $n\geq 2$.
\item $\mathbb{T} $, the tetrahedral group of $12$ rotational symmetries of a tetrahedron.
\item $\mathbb{O} $, the octahedral group of $24$ rotational symmetries of an octahedron.
\item $\mathbb{I} $, the icosahedral group of $60$ rotational symmetries of an icosahedron.
\end{itemize}

Since all elements of order $2$ in $H_1$ are elements in $G_1$, we see that $G_1 \cong \mathbb{T} $ or $G_1 \cong \mathbb{Z}_n $ for some $n\geq 2$. If $G_1 \cong \mathbb{T} $, then $[G_1 : H_1]=3 $, and thus $[G_2 : H_2]=[\hat{G} : H_1 \times H_2 ]=[G_1 : H_1]=3 $. Consequently, if $G_1$ is not a cyclic group, $|G_2|$ must be odd. Moreover, it's evident that $G_2$ is a cyclic group when $|G_2|$ is odd. So, within these two groups $G_1$ and $G_2$, there must be at least one cyclic group. 

Without loss of generality, we assume that $G_2$ is a cyclic group. Then $G_2$ is generated by a rotation around some axis. Consequently, there exists an embedding $\tau : SO(2) \to (\{ {\rm id_3} \} \times SO(3)) $ such that $G_2 < \tau (SO(2)) $ and $\tau (SO(2))$ commutes with $\hat{G}$. By lifting the action of $\tau (SO(2))$ on $\mathbb{R}P^3 \cong SO(3) $ to $\mathbb{S}^3$, we can obtain a fixed point free $SO(2)$-action on $\mathbb{S}^3$ that preserves the metric, and this action commutes with the action of $G$. Hence the $SO(2)$-action gives a complex structure on $\mathbb{R}^4$, and $G< U(2)$, where $U(2)$ is the unitary group corresponding to this complex structure. This completes the proof.
\end{proof}


\begin{thebibliography}{BBEGZ16} 
{

\bibitem{anderson1} M. T.~Anderson:
\emph{Convergence and rigidity of manifolds under Ricci curvature bounds},
Invent. Math. $\mathbf{102}$ (1990), 429–445.

\bibitem{bkn1} S. Bando, A. Kasue, H. Nakajima:
\emph{On a construction of coordinates at infinity on manifolds with fast curvature decay and maximal volume growth},
Invent. Math. $\mathbf{97}$ (1989), 313–349.

\bibitem{bhpv1} W. P. Barth, K. Hulek, C. A. M. Peters and A. Ven:
\emph{Compact complex surfaces},
second edition, Springer-Verlag, Berlin, 2004. xii+436 pp.

\bibitem{be1}  A. L. Besse:
\emph{Einstein Manifolds},
Reprint of the 1987 edition. Classics in Mathematics. Springer-Verlag, Berlin, 2008. xii+516 pp.

\bibitem{biq1}  O. Biquard:
\emph{D\'esingularisation de m\'etriques d’Einstein. I},
Invent. Math. $\mathbf{192}$ (2013), no.1, 197–252.

\bibitem{biq2}  O. Biquard:
\emph{D\'esingularisation de m\'etriques d’Einstein. II},
Invent. Math. $\mathbf{204}$ (2016), no.1, 473–504.

\bibitem{bps1}  E. Bru\`e, A. Pigati, D. Semola:
\emph{Topological regularity and stability of noncollapsed spaces with Ricci curvature bounded below},
preprint, arxiv: 2405.03839.

\bibitem{chco3} J.~Cheeger, T. H. Colding:
\emph{On the structure of spaces with Ricci curvature bounded below. I},
J. Differential Geom. $\mathbf{46}$ (1997), 406–480.

\bibitem{jcwsjan1} J.~Cheeger, W.-S. Jiang, A.~Naber:
\emph{Rectifiability of singular sets of noncollapsed limit spaces with Ricci curvature bounded below},
Ann. of Math. $\mathbf{193}$ (2021), 407–538. 

\bibitem{colding1} T. H. Colding:
\emph{Ricci Curvature and Volume Convergence},
Ann. of Math. $\mathbf{145}$ (1997), 477-501.

\bibitem{coldingnaber2} T.-H. Colding, A.~Naber:
\emph{Characterization of tangent cones of noncollapsed limits with lower Ricci bounds and applications},
Geom. Funct. Anal. $\mathbf{23}$ (2013), 134–148.

\bibitem{rhamilton1} R.~Hamilton:
\emph{Three-manifolds with positive Ricci curvature},
J. Differential Geom. $\mathbf{17}$ (1982), 255–306.

\bibitem{hnw1} E. Hupp, A. Naber, K.-H. Wang: 
\emph{Lower Ricci Curvature and Nonexistence of Manifold Structure},
preprint, arxiv: 2308.03909, to appear in Geom. Topol.

\bibitem{klein1} E. Klein:
\emph{Lectures on the icosahedron and the solution of equations of the fifth degree},
Dover Publications, Inc., New York, N.Y., revised edition, 1956.

\bibitem{krhm1} P. B. Kronheimer:
\emph{The construction of ALE spaces as hyper-K\"ahler quotients},
J. Differential Geom. $\mathbf{29}$ (1989), 665–683.

\bibitem{kl1} K. Lamotke:
\emph{Regular Solids and Isolated Singularities},
Adv. Lectures Math. Friedr. Vieweg \& Sohn, Braunschweig, 1986. x+224 pp.

\bibitem{lg1} G. Liu: 
\emph{Complete K\"ahler manifolds with nonnegative Ricci curvature},
preprint, arxiv: 2404.08537.

\bibitem{menguy2} X. Menguy:
\emph{Noncollapsing examples with positive Ricci curvature and infinite topological type},
Geom. Funct. Anal. $\mathbf{10}$ (2000), 600–627.

\bibitem{ozuch1} T. Ozuch:
\emph{Noncollapsed degeneration of Einstein 4-manifolds I.},
Geom. Topol. $\mathbf{26}$ (2022), 1483–1528.

\bibitem{ozuch2} T. Ozuch:
\emph{Noncollapsed degeneration of Einstein 4-manifolds II.},
Geom. Topol. $\mathbf{26}$ (2022), 1529-1634.

\bibitem{pp1} P.~Petersen:
\emph{Riemannian geometry},
3rd ed., Grad. Texts in Math, 171. Springer, Cham, 2016.

\bibitem{perel1} G.~Perelman:
\emph{Construction of manifolds of positive Ricci curvature with big volume and large Betti numbers},
Comparison geometry (Berkeley, CA, 1993-94), 157-163, Math. Sci. Res. Inst. Publ., 30, Cambridge Univ. Press, Cambridge, 1997.

\bibitem{egrs1} E. G. Rees:
\emph{Notes on Geometry},
Universitext in Mathematics. Berlin-Heidelberg-New York: Springer 1983.

\bibitem{sus1} I. \c Suvaina:
\emph{ALE Ricci-flat K\"ahler metrics and deformations of quotient surface singularities},
Ann. Global Anal. Geom. $\mathbf{41}$ (2012), 109-123.

\bibitem{thur1} W. Thurston:
\emph{Three-dimensional geometry and topology. Vol. 1},
Princeton Mathematical Series, 35. Princeton University Press, Princeton, New Jersey, 1997.

\bibitem{tg5} G.~Tian:
\emph{On Calabi's conjecture for complex surfaces with positive first Chern class},
Invent. Math. $\mathbf{101}$ (1990), 101–172.

\bibitem{wolf1} J. A. Wolf:
\emph{Spaces of constant curvature},
AMS Chelsea Publishing, Providence, RI, 2011. xviii+424 pp.

\bibitem{wright1} E. P. Wright:
\emph{Quotients of gravitational instantons},
Ann. Global Anal. Geom. $\mathbf{41}$ (2012), 91-108.

}\end{thebibliography}
\end{document}